\documentclass[11pt, a4paper]{amsart}
\usepackage{amssymb, amsmath, amscd, bm, amsthm, pdfpages, setspace, hyperref, mathtools, subcaption, graphicx, enumerate, cite}

\theoremstyle{plain}
\newtheorem{theorem}{Theorem}[section]
\newtheorem{proposition}[theorem]{Proposition}
\newtheorem{definition}[theorem]{Definition}

\newtheorem{lemma}[theorem]{Lemma}

\newtheorem{corollary}[theorem]{Corollary}

\newcommand{\re}{\mathop{\rm Re}\nolimits}
\newcommand{\im}{\mathop{\rm Im}\nolimits}

\newcommand{\ran}{\mathop{\rm ran}\nolimits}

\def\@Rref#1{\hbox{\rm \ref{#1}}}
\def\Rref#1{\@Rref{#1}}

\theoremstyle{plain}

\begin{document}

\title[Decay of Operator Semigroups]{Decay of Operator Semigroups, Infinite-time Admissibility, and Related 
Resolvent Estimates}

\thispagestyle{plain}

\author{Masashi Wakaiki}
\address{Graduate School of System Informatics, Kobe University, Nada, Kobe, Hyogo 657-8501, Japan}
 \email{wakaiki@ruby.kobe-u.ac.jp}
 \thanks{This work was supported by JSPS KAKENHI Grant Number JP20K14362.}

\begin{abstract}
We study decay rates for bounded $C_0$-semigroups
from the perspective of $L^p$-infinite-time admissibility and related 
resolvent estimates. In the Hilbert space setting,
polynomial decay of semigroup orbits is characterized by
the resolvent behavior in the open right half-plane.
A similar characterization based on $L^p$-infinite-time admissibility 
is provided for 
multiplication semigroups on $L^q$-spaces with $1 \leq q \leq p < \infty$.
For
polynomially stable $C_0$-semigroups on Hilbert spaces,
we also give
a sufficient condition for $L^2$-infinite-time admissibility.
\end{abstract}


\maketitle

\section{Introduction}
Consider the abstract Cauchy problem
\[
\begin{cases}
\dot u (t) = Au(t), & t \geq 0, \\
u(0) = x,& x \in X,
\end{cases}
\]
where $A$ is the generator of a bounded $C_0$-semigroup $(T(t))_{t\geq 0}$
on a Banach space $X$. Throughout this paper,
we consider the domain $D(A)$ of $A$ to be 
equipped with 
the graph norm of $A$.
Let us assume that $A$ is invertible.
Then the domain of $A$ coincides with the range of $A^{-1}$.
To obtain uniform decay rates of classical solutions, 
we study
the quantitative behavior of the operator norm $\|T(t)A^{-1}\|$
as $t \to \infty$.
We concentrate mainly on the situation where $\|T(t)A^{-1}\|$
decays polynomially, i.e., $\|T(t)A^{-1}\| = O(t^{-1/\alpha})$ as 
$t \to \infty$ for some $\alpha >0$.

The relation between polynomial rates of decay of $\|T(t)A^{-1}\|$ and 
growth of $\|(i\eta I-A)^{-1}\|$ 
on the imaginary axis $i \mathbb{R}$ was studied, e.g., in \cite{Liu2005PDR, Batkai2006, Batty2008, Borichev2010}.
In particular, it was shown in \cite{Borichev2010} that 
$\|T(t)A^{-1}\| = O(t^{-1/\alpha})$ as
$t \to \infty$ is equivalent to 
$\|(i\eta I - A)^{-1}\| = O(|\eta|^{\alpha})$
as $|\eta| \to \infty$ in the Hilbert space setting.
We refer to \cite{Batty2008, Batty2016, Rozendaal2019, Chill2020}  for 
more general transference between semigroup decay and 
resolvent growth on $i\mathbb{R}$.

The long-time asymptotic behavior of semigroup orbits is 
related to the resolvent of $A$ on the 
open right half-plane $\mathbb{C}_+$ as well. For example, 
the Gearhart-Pr\"uss theorem (see, e.g., \cite[Theorem~5.2.1]{Arendt2001} and 
\cite[Theorem~V.1.11]{Engel2000})
shows that
exponential stability of a $C_0$-semigroup on a Hilbert space
is equivalent to uniform boundedness of the resolvent on $\mathbb{C}_+$.
Moreover, it was proved in \cite{Tomilov2001} that density of the set
\[
\left\{
x \in X :
\lim_{\xi \to 0+} 
\sqrt{\xi} \big((\xi+i\eta)I - A \big)^{-1}x = 0 \text{~for all $\eta \in \mathbb{R}$}
\right\}
\]
is sufficient for strong stability of the bounded $C_0$-semigroup when $X$ is a Hilbert space.
See the survey papers \cite{Chill2007,Chill2020}
for further references and recent developments on stability of $C_0$-semigroups.
In
\cite{Zwart2003_Ulmer,Eisner2006, Eisner2007, Rozendall2018, Helffer2021,Arnold2022,Arnold2023},  the generator $A$ 
of a (not necessarily bounded) $C_0$-semigroup $(T(t))_{t \geq 0}$
was considered, and
upper bounds for non-exponential
growth rates of $\|T(t)\|$ were obtained from
the so-called $\alpha$-Kreiss condition: The spectrum $\sigma(A)$ of $A$
is contained in the closed left half-plane $\overline{\mathbb{C}}_-$ and
there exist $K \geq 1$ and $0\leq \alpha < \infty$ such that
$\|(\lambda I - A)^{-1}\| \leq K((\re \lambda)^{-\alpha} + 1)$
for all $\lambda \in \mathbb{C}_+$. Some
variations of these results were reported in \cite{Neerven2009, Boukdir2015}.
The estimate $\|T(t)\| = O(t/\sqrt{\log t})$, proved in \cite{Arnold2022}, 
is the best obtained so far from the $1$-Kreiss condition for 
Hilbert space semigroups.

For
a bounded $C_0$-semigroup $(T(t))_{t\geq 0}$ on a Hilbert space such that 
$\sigma(A) \cap i \mathbb{R} = \emptyset$,
$\|T(t)A^{-1}\| = O(t^{-1/\alpha})$ 
if and only if $\sup_{\lambda \in \mathbb{C}_+}
\|(\lambda I - A)^{-1}(-A)^{-\alpha}\| < \infty$; see \cite{Borichev2010}.
In this equivalence, the fractional power
$(-A)^{-\alpha}$ plays a role of a smoothing factor that cancels
resolvent growth.
When we replace it with a weaker
smoothing factor 
$(-A)^{-\alpha+ \varepsilon}$, $0< \varepsilon<\alpha$, the resulting norm
$\|(\lambda I - A)^{-1}(-A)^{-\alpha+\varepsilon}\|$ may  
grow to infinity as $\re \lambda \to 0 +$.
Then it is natural to ask whether this kind of resolvent growth 
on $\mathbb{C}_+$
also characterizes
the decay of $\|T(t)A^{-1}\|$ as $t\to \infty$.

We study
a resolvent estimate on $\mathbb{C}_+$,
sometimes called the $p$-Weiss condition in the community
of infinite-dimensional systems. 
Let $Y$ be another Banach space, and consider $C \in \mathcal{L}(D(A),Y)$,
i.e., a linear bounded operator $C$ from $D(A)$ to $Y$.
The $p$-Weiss condition on $C$ is defined
as follows.
\begin{definition}
	\label{def:Weiss_cond}
	{\em
		Let $X$ and $Y$ be Banach spaces, 
		let $A$ be the generator of a bounded $C_0$-semigroup 
		$(T(t))_{t \geq 0}$ on $X$, and let $1 \leq p \leq \infty$.
		An operator $C \in \mathcal{L}(D(A),Y)$ {\em satisfies the $p$-Weiss condition for $A$} 
		if there exists $K>0$ such that for all $\lambda \in \mathbb{C}_+ $,
		\[
		\|CR(\lambda,A)\| \leq \frac{K}{(\re \lambda)^{1-1/p}},
		\]
		where $1/p \coloneqq 0$ for $p = \infty$.
	}
\end{definition}
Following \cite{Weiss1989_observation}, we introduce the notion of admissibility.
The $p$-Weiss condition was derived as a necessary condition for
this notion by Weiss \cite{Weiss1991}.
\begin{definition}
	{\em
		Let $X$ and $Y$ be Banach spaces, 
		let $A$ be the generator of a $C_0$-semigroup 
		$(T(t))_{t \geq 0}$ on $X$, and let $1 \leq p < \infty$.
		\begin{itemize}
			\item An operator $C \in \mathcal{L}(D(A),Y)$ is {\em $L^p$-infinite-time admissible for $A$}
			if there exists $M >0$ such that 
			\begin{equation}
			\label{eq:inf_time_admissible}
			\displaystyle \int^{\infty}_0 \|CT(t)x\|^p dt \leq M \|x\|^p
			\end{equation}
			for all $x \in D(A)$.
			\item 
			An operator $C \in \mathcal{L}(D(A),Y)$ is {\em $L^{\infty}$-infinite-time admissible for $A$
				if there exists $M >0$ such that 
				$\sup_{t \geq 0} \|CT(t)x\| \leq M \|x\|$ for all $x \in D(A)$.}
		\end{itemize}
	}
\end{definition}
Weiss conjectured in \cite{Weiss1991, Weiss1998} that 
the $2$-Weiss condition and $L^2$-infinite-time admissibility are equivalent
when $X$ and $Y$ are Hilbert spaces. 
This conjecture was resolved negatively.
Counterexamples can be found in \cite{Jacob2002,
	Zwart2003, Jacob2004}.
However, positive results on the equivalence were obtained in
several situations:  a) $(T(t))_{t \geq 0}$ is an exponentially stable and 
right-invertible $C_0$-semigroup \cite{Weiss1991}; b) $(T(t))_{t \geq 0}$ is a contraction $C_0$-semigroup
and $Y$ is a finite-dimensional space \cite{Jacob2001}; and  c) 
$(T(t))_{t \geq 0}$ is a bounded analytic $C_0$-semigroup 
such that $(-A)^{1/2}$ is $L^2$-infinite-time admissible for $A$ \cite{LeMerdy2003}.
While $X$ and $Y$ are Hilbert  spaces  in the results a) and b), they are Banach spaces
in the result c).
Moreover, the result c) was extended from the $L^2$-case to the $L^p$-case for
$1 \leq p \leq \infty$
in
\cite{Haak2007, Bounit2010} and to Orlitz spaces in \cite{Hosfeld2023}.
For more information on admissibility and related resolvent conditions,
we refer to the survey article \cite{Jacob2004Survey} and the books
\cite{Staffans2005, Tucsnak2009}.

For $C \in \mathcal{L}(D(A),Y)$,
the relation of
the $p$-Weiss condition and $L^p$-infinite-time admissibility to
the decay rate of the form
\begin{equation}
\label{eq:CT_decay_square}
\|CT(t)\| \leq \frac{M}{t^{1/p}}\qquad \text{as $t \to \infty$ for some
	$M >0$}
\end{equation}
was studied in \cite{Bounit2010,Haak2012, Zwart2012, Hosfeld2023}, where $CT(t)$ extends an operator in $\mathcal{L}(X,Y)$
(again denoted by $CT(t)$)  for all $t >0$.
For bounded analytic semigroups on Banach spaces,
the decay estimate in \eqref{eq:CT_decay_square} with $1 < p \leq \infty$
holds if and only if
the $p$-Weiss condition is satisfied
\cite{Bounit2010,Haak2012}.
This
was extended to decay rates of a more general class of functions
in \cite{Hosfeld2023}. 
It was shown  in \cite{Zwart2012} that 
if a bounded $C_0$-semigroup $(T(t))_{t \geq 0}$ 
on a Hilbert space commutes with $C$ for all
$t \geq 0$, then
$L^2$-infinite-time admissibility implies
the decay rate in \eqref{eq:CT_decay_square} with $p=2$.

While most previous studies on admissibility focus on
$C \in \mathcal{L}(D(A),Y)$, we sometimes 
deal with $C \in \mathcal{L}(X,Y)$ in this paper.
This might seem unconventional, because $C \in \mathcal{L}(X,Y)$
is infinite-time admissible for any generator of
an exponentially stable $C_0$-semigroup.
However, 
it is not trivial that when the $C_0$-semigroup have weaker stability properties, 
$C \in \mathcal{L}(X,Y)$ satisfies $L^p$-infinite-time admissibility and the $p$-Weiss condition with $1 \leq p < \infty$.
In particular, we are interested in the case where $C$ is a fractional power  $(-A)^{-\delta}$ 
for some $\delta >0$, which is not only bounded on $X$
but also has a certain smoothing effect.
This is in contrast to the situation of 
the above-mentioned previous studies \cite{LeMerdy2003, Haak2007, Bounit2010,Hosfeld2023}, which discussed
the infinite-time admissibility
of $(-A)^{\delta}$ with $\delta >0$
as
a condition for the Weiss conjecture to be true for bounded analytic $C_0$-semigroups.

For fixed $\alpha >0$ and $1 \leq p < \infty$,
we prove that $(-A)^{-\alpha/p}$ satisfies the $p$-Weiss condition for $A$
if and only if $\|T(t)A^{-1}\| = O(t^{-1/\alpha})$ as $t \to \infty$,
where $(T(t))_{t \geq 0}$ is a bounded $C_0$-semigroup on a Hilbert space
such that $0 \notin \sigma(A)$.
Note that the $p$-Weiss condition on $(-A)^{-\alpha/p}$ 
gives an upper bound for the rate of growth of
$\|(\lambda I- A)^{-1}(-A)^{-\alpha/p}\|$ as $\re \lambda \to 0+$.
Since $(-A)^{-\alpha/p}$ is bounded, the Hille-Yosida theorem guarantees
a faster decay rate 
$\|(\lambda I- A)^{-1}(-A)^{-\alpha/p}\| = O(1/\re\lambda)$ as $\re \lambda 
\to \infty$ than the $p$-Weiss condition.
We also show that
non-polynomial decay of the form  $\|T(t)(-A)^{-\alpha}\| 
= O(t^{-\beta} (\log t)^{-\gamma})$, 
where $\alpha >0$, $0\leq \beta < 1$,
and $\gamma \geq 0$,
can be  characterized by resolvent growth on $\mathbb{C}_+$.
Following \cite{Martinez2011,
	Seifert2015KTthm, Chill2016, Batty2016, Rozendaal2019}, 
we obtain analogous characterizations of 
rates of decay for $\|T(t)A(I-A)^{-1}\|$ and  $\|T(t)A(I-A)^{-2}\|$.
They are motivated by the problem of quantifying
uniform rates of decay of orbits  starting in
the range $\ran (A)$ of $A$ and in $D(A)\cap \ran (A)$, respectively.

In the Banach space setting, we can easily see that for 
a bounded $C_0$-semigroup $(T(t))_{t \geq 0}$,
the decay property $\|T(t)A^{-1}\| = O(t^{-1/\alpha})$ is almost
equivalent to the $L^p$-infinite-time admissibility of $(-A)^{-\alpha/p}$.
More precisely, if $(-A)^{-\alpha/p}$ is 
$L^p$-infinite-time admissible for $A$, then
$\|T(t)A^{-1}\| = O(t^{-1/\alpha})$ as $t \to \infty$. Conversely,
if this decay estimate is satisfied, then
$(-A)^{-\alpha/p-\varepsilon}$ with $\varepsilon>0$ 
is $L^p$-infinite-time admissible for $A$.
We do not know whether the latter implication with 
$\varepsilon = 0$ is true for
general bounded $C_0$-semigroups. However, 
for multiplication semigroups on $L^q$-spaces with $1 \leq q \leq p < \infty$,
it is true and hence we obtain an equivalence result.

Let $X$ and $Y$ be Hilbert spaces, and 
consider a bounded $C_0$-semigroup $(T(t))_{t \geq 0}$
on $X$ such that $\|T(t)A^{-1}\|$ decays polynomially.
The next objective is to obtain a sufficient condition for
$C \in \mathcal{L}(D(A),Y)$ to be $L^2$-infinite-time admissible for $A$.
One of the motivations is that the dual notion of $L^2$-infinite-time
admissibility in Definition~\ref{def:Weiss_cond} is 
closely related to input-to-state stability with respect to
squared integrable inputs; see \cite{Weiss1989_observation} for duality results on
admissibility. 
A survey on input-to-state stability of infinite-dimensional systems is given 
in \cite{Mironchenko2020}.
A crucial part of the sufficient condition is that 
there exists $\alpha >0$ such that 
$C(-A)^{\alpha}$ extends to an operator in $\mathcal{L}(D(A),Y)$
and its extension satisfies the $2$-Weiss condition for $A$, which is called
the {\em strong $2$-Weiss condition for $A$} in this paper.
We show that $C$ is $L^2$-infinite-time admissible if
$C$ satisfies the strong $2$-Weiss condition and is $L^2$-{\em finite}-time
admissible, in whose definition the integral over $[0,\infty)$
in \eqref{eq:inf_time_admissible} is replaced by an integral over 
some finite interval $[0,t_1)$.
We also examine the relation between the strong $2$-Weiss condition and 
the decay rate of the form 
$\|CT(t)\| = O(1/\sqrt{t^{1+\beta}})$ for some $\beta >0$ when $Y = X$
and $T(t)$ commutes with $C$ for all $t \geq 0$.

The paper is organized as follows.
In Section~\ref{sec:Preliminaries}, we recall some basics on the
polynomial decay of semigroup orbits and Plancherel's theorem.
In Section~\ref{sec:polynomial}, polynomial  rates of orbit decay
are characterized by the Weiss condition in the Hilbert space setting and
by infinite-time admissibility in the Banach space setting.
In Section~\ref{sec:non_polynomial}, we establish 
an analogous resolvent characterization
of the property $\|T(t)(-A)^{-\alpha}\| 
= O(t^{-\beta} (\log t)^{-\gamma})$, 
where $\alpha >0$, $0\leq \beta < 1$,
and $\gamma \geq 0$.
In Section~\ref{sec:suff_cond}, we obtain a sufficient condition for
$L^2$-infinite-time admissibility by using the decay estimate of $\|T(t)(-A)^{-\alpha}\|$.

\subsection*{Notation}
Let $\mathbb{C}_+$, $\mathbb{C}_-$, and $i\mathbb{R}$ denote the open right half-plane
$\{\lambda \in \mathbb{C} :\re \lambda >0 \}$,
the open left half-plane
$\{\lambda \in \mathbb{C} : \re \lambda  <0 \}$,
and the imaginary axis
$\{i \eta : \eta \in \mathbb{R}\}$, respectively.
Let $X$ and $Y$ be Banach spaces. 
The space of
all bounded linear operators from $X$ to $Y$ is denoted by
$\mathcal{L}(X,Y)$. We write $\mathcal{L}(X) \coloneqq \mathcal{L}(X,X)$.
The domain and the range of a linear operator $A\colon X \to Y$ are denoted by $D(A)$ and
$\ran(A)$, respectively.
When $A$ is closed, 
$D(A)$ is seen as a Banach space with the graph norm $\|x\|_A \coloneqq \|x\| + \|Ax\|$.
We denote by $\sigma(A)$ and $\varrho(A)$
the spectrum and 
the resolvent set of a linear operator $A\colon D(A) \subset X \to X$, respectively.
We write
$R(\lambda,A) \coloneqq (\lambda I - A)^{-1}$ for $\lambda \in \varrho(A)$.
For linear operators $A\colon D(A) \subset X \to X $ and $B\colon D(B) \subset X \to X$, we write
$A \subset B$ if $D(A) \subset D(B)$ and $Ax = Bx$ for all $x \in D(A)$,
and the composite $BA$ is defined by $BAx \coloneqq B(Ax)$ with domain
$D(BA) \coloneqq \{
x \in D(A) : Ax \in D(B)
\}$.
Let
$\langle \cdot , \cdot \rangle$ denote
the inner product in a Hilbert space.

\section{Preliminaries}
\label{sec:Preliminaries}
\subsection{Fractional powers and polynomial decay}
Let $X$ be a Banach space.
For a linear operator $A\colon D(A) \subset X \to X$
and $T \in \mathcal{L}(X)$,
we say that 
$T$ {\em commutes with} $A$ if 
$
TA  \subset AT
$.
We say that a densely defined, linear operator $A$ on $X$
is {\em sectorial} if $(-\infty,0) \subset \varrho(A)$ and if there exists $K >0$ such that 
$\|\lambda(\lambda+A)^{-1}\| \leq K$ for all $\lambda >0$.
We refer to \cite[Section~2.1]{Haase2006} for background information 
on sectorial operators.

Let $(T(t))_{t \geq0}$ be a bounded $C_0$-semigroup 
on a Banach space $X$ with generator $A$. The negative generator 
$-A$ is sectorial, and the fractional power $(-A)^{\alpha}$ is well defined for each 
$\alpha >0$.
If $0 \in \varrho(A)$, then $-A^{-1}$ is also sectorial and  for all $\alpha \geq 0$, $(-A)^{-\alpha}$ is well defined and satisfies $(-A)^{-\alpha} = 
(-A^{-1})^{\alpha} = ((-A)^{\alpha})^{-1}$.
Since 
$R(\lambda,A)$ and $T(t)$ commute with $A$
for all $\lambda \in \varrho(A)$ and $t \geq 0$,
they also commute with $(-A)^\alpha$ for all $\alpha \in \mathbb{R}$; 
see, e.g., Propositions~3.1.1.f) and 3.2.1.a) of \cite{Haase2006}.
We shall use this commutative property frequently without comment.
Most properties of fractional powers of sectorial operators needed in this paper
can be found in 
\cite[Chapter~3]{Haase2006}.

For a fixed $\alpha>0$, 
the decay rate of $\|T(t)(-A)^{-\alpha}\|$ 
is linked to that of $\|T(t)A^{-1}\|$ in the Banach space setting as shown in
\cite[Proposition~3.1]{Batkai2006}.
We here present it in a general form. The proof can be found in
\cite[Lemma~4.2]{Batty2016}.
\begin{proposition}
	\label{prop:interpolation}
	Let $(T(t))_{t\geq 0}$ be a bounded $C_0$-semigroup  on
	a Banach space $X$, and let $B \in \mathcal{L}(X)$
	be a sectorial operator such that $T(t)$ commutes with $B$
	for all $t \geq 0$. 
	For fixed $\alpha ,\beta >0$,  the following statements are equivalent:
	\begin{enumerate}[\em (\roman{enumi})]
		\item 
		$\|T(t)B\| = O(t^{-1/\alpha})$ as $t \to \infty$.
		\item 
		$\|T(t)B^\beta\| = O(t^{-\beta/\alpha})$ as $t \to \infty$.
	\end{enumerate}
\end{proposition}

Polynomial rates of decay of bounded $C_0$-semigroups on Hilbert spaces
can be characterized by uniform boundedness of the resolvent
on $\mathbb{C}_+$; see Lemma~2.3 and Theorem~2.4 of \cite{Borichev2010} 
for the proof.
\begin{theorem}
	\label{thm:uniform_bounded}
	Let
	$(T(t))_{t \geq 0}$ be a bounded $C_0$-semigroup on
	a Hilbert space $X$ with generator $A$ such that $i \mathbb{R} \subset \varrho(A)$.
	For a fixed $\alpha >0$, the following statements are equivalent:
	\begin{enumerate}[\em (\roman{enumi})]
		\item 
		$\|T(t)A^{-1}\| = O(t^{-1/\alpha})$ as $t \to \infty$.
		\item 
		There exists $M>0$ such that $\|R(\lambda,A) (-A)^{-\alpha}\| \leq M$
		for all $\lambda \in \mathbb{C}_+$.	
	\end{enumerate}
\end{theorem}

\subsection{Plancherel's theorem}
Let $X$ and $Y$ be Hilbert spaces.
Let $(T(t))_{t \geq 0}$ be a bounded $C_0$-semigroup on 
$X$ with generator $A$, and let $C \in \mathcal{L}(D(A),Y)$.
For fixed $x \in X$, $y \in D(A)$, and $\xi>0$, define the functions $f\colon \mathbb{R} \to X$ and
$g\colon \mathbb{R} \to Y$ by
\begin{align*}
f(t) \coloneqq 
\begin{cases}
e^{-\xi t} T(t)x, & t \geq 0, \\
0, & t < 0,
\end{cases}\qquad
g(t) \coloneqq 
\begin{cases}
te^{-\xi t} CT(t)y, & t \geq 0, \\
0, & t < 0.
\end{cases}
\end{align*}
Their Fourier transforms $\mathcal{F} f$ and $\mathcal{F} g$  
are given by $(\mathcal{F} f)(\eta) = R(\xi+i\eta,A)x$ and 
$(\mathcal{F} g)(\eta) = CR(\xi +i\eta)^2y$ for 
$\eta \in \mathbb{R}$.
Applying
Hilbert-space-valued Plancherel's theorem (see, e.g., \cite[Theorem~1.8.2]{Arendt2001}
and \cite[Theorem~C.14]{Engel2000}) to $f$
and $g$, we obtain
the next result, which is also called Plancherel's theorem throughout this paper.
\begin{theorem}
	Let $X$ and $Y$ be Hilbert spaces.
	Let $(T(t))_{t \geq 0}$ be a bounded $C_0$-semigroup on
	a Hilbert space $X$ with generator $A$, and let $C \in \mathcal{L}(D(A),Y)$.
	For all $x \in X$, $y \in D(A)$, and $\xi>0$, 
	\begin{align*}
	\int^{\infty}_0 \|e^{-\xi t} T(t)x\|^2 dt &= 
	\frac{1}{2\pi} \int^{\infty}_{-\infty} \| R(\xi +i\eta,A)x\|^2 d\eta, \\
	\int^{\infty}_0 \|t e^{-\xi t} CT(t)y\|^2 dt &= 
	\frac{1}{2\pi} \int^{\infty}_{-\infty} \| CR(\xi +i\eta,A)^2y\|^2 d\eta.
	\end{align*}
\end{theorem}

\section{Polynomial decay of semigroup orbits}
\label{sec:polynomial}
In this section,
we study the relation of the Weiss condition and infinite-time admissibility to
polynomial rates of decay of semigroup orbits.

\subsection{Decay of Hilbert space semigroups and the Weiss condition}
The $p$-Weiss condition on $(-A)^{-\alpha/p}$ characterizes the decay of 
$\|T(t)A^{-1}\|$
in the Hilbert space context, which is the main result of this section.
The case $p=1$  was obtained in Lemma~2.3 and Theorem~2.4 of \cite{Borichev2010}, which
is stated as 
Theorem~\ref{thm:uniform_bounded} in this paper.
\begin{theorem}
	\label{thm:decay_rate_equivalent}
	Let $(T(t))_{t \geq 0}$ be a bounded $C_0$-semigroup on a Hilbert 
	space $X$ with generator $A$ such that $0 \in  \varrho(A)$.
	For fixed $\alpha >0$ and $1 \leq p < \infty$, the following statements are equivalent:
	\begin{enumerate}[\em (\roman{enumi})]
		\item 
		$\|T(t) A^{-1}\| = O(t^{-1/\alpha})$ as $t\to \infty$.
		\item 
		$(-A)^{-\alpha/p}$ satisfies the $p$-Weiss condition for $A$.
	\end{enumerate}
\end{theorem}

Before proving Theorem~\ref{thm:decay_rate_equivalent}, we present related results.
As a corollary of Theorem~\ref{thm:decay_rate_equivalent},
we separately state
an interpolation property for the generator of a bounded $C_0$-semigroup 
on a Hilbert space,
which is the resolvent analogue of Proposition~\ref{prop:interpolation} with $B\coloneqq (-A)^{-1}$.
\begin{corollary}
	Let $(T(t))_{t \geq 0}$ be a bounded $C_0$-semigroup on a Hilbert 
	space $X$ with generator $A$ such that $0 \in \varrho(A)$.
	For fixed $\alpha> \beta >0$, the following statements are equivalent: 
	\begin{enumerate}[\em (\roman{enumi})]
		\item 
		There exists $M_\alpha>0$ such that
		$\|R(\lambda,A) (-A)^{- \alpha}\| \leq M_\alpha$ for all $\lambda \in
		\mathbb{C}_+$.
		\item 
		There exists $M_{\alpha,\,\beta}>0$ such that
		\[
		\|R(\lambda,A) (-A)^{-\beta}\| \leq \frac{M_{\alpha,\,\beta}}{(\re \lambda)^{1-\beta/\alpha}}
		\]
		for all $\lambda \in
		\mathbb{C}_+$.
	\end{enumerate}
\end{corollary}

Theorem~\ref{thm:decay_rate_equivalent} provides
a resolvent characterization of rates of decay for the orbits $T(t)x$ with $x \in D(A)$.
We will give analogous results on
decay rates for the orbits $T(t)x$ with $x \in \ran(A)$ and 
$x \in D(A) \cap \ran(A)$.
We begin by presenting some background materials on these orbits.

In the case $x \in \ran(A)$, 
we are interested in the orbits of the form $T(t)Ay$ for $y \in D(A)$ and hence 
study the decay rate of $\|T(t)A(I-A)^{-1}\|$.
Since $D(A) \cap \ran(A) =  \ran(A (I-A)^{-2})$ by 
\cite[Proposition~3.10]{Batty2016},
the problem we address in the case $x \in D(A) \cap \ran(A)$ is to quantify
the decay of $\|T(t)A(I-A)^{-2}\|$.
By \cite[Lemma~3.2]{Batty2016}, $-A(I-A)^{-1}$  and $-A(I-A)^{-2}$
are sectorial. Therefore, the fractional powers
$(-A(I-A)^{-1})^{\alpha}$ and
$(-A(I-A)^{-2})^{\alpha}$ are well defined for all $\alpha >0$. 
Using the product and composition rules (see, e.g., 
Theorem 3.7 (iv) and Remark 3.8 (iv) of \cite{Batty2016}),
we obtain
\[
(-A(I-A)^{-k})^{\alpha} = (-A)^{\alpha}(I-A)^{-k\alpha}
\]
for each $\alpha >0$ and $k \in \{1,2  \}$.

When $0 \in \varrho(A)$,
the decay of $\|T(t)A(I-A)^{-1}\|$ to zero is equivalent to 
exponential stability of $(T(t))_{\geq 0}$, and
the decay  rate of $\|T(t)A(I-A)^{-2}\|$
is the same as that of $\|T(t)A^{-1}\|$.
Therefore, we assume that $0 \in \sigma(A)$. 
Moreover,
we concentrate on the situation where 
$\|T(t)A(I-A)^{-1}\|$ and $\|T(t)A(I-A)^{-2}\|$  decay no faster than $t^{-1}$ as $t \to \infty$.
If not, $0$ is the eigenvalue of $A$ and an isolated point of $\sigma(A)$,
which is not in our interest; see Theorem~6.9 and a paragraph before Theorem~8.1 
of \cite{Batty2016} for details. 

We provide a characterization of decay rates for
$\|T(t)A(I-A)^{-1}\|$  and $\|T(t)A(I-A)^{-2}\|$ by the $p$-Weiss condition.
This characterization in the case $p=1$ was obtained in 
Corollary~7.5 and Theorem~7.6 of \cite{Batty2016} for $\|T(t)A(I-A)^{-1}\|$
and in
Theorem~8.4 and its proof of \cite{Batty2016}
for $\|T(t)A(I-A)^{-2}\|$.
The proof will be given after that of Theorem\ref{thm:decay_rate_equivalent}.
\begin{theorem}
	\label{thm:singularity_zero}
	Let $(T(t))_{t \geq 0}$ be a bounded $C_0$-semigroup on a Hilbert 
	space $X$ with generator $A$ such that $0 \in \sigma(A)$.
	For fixed $\alpha \geq 1$, $1 \leq p < \infty$, and $k \in \{1,2\}$, 
	the following statements are equivalent:
	\begin{enumerate}[\em (\roman{enumi})]
		\item 
		$\|T(t) A(I-A)^{-k}\| = O(t^{-1/\alpha})$ as $t\to \infty$.
		\item 
		$(-A)^{\alpha/p}(I-A)^{-k\alpha/p}$ satisfies the $p$-Weiss condition for $A$.
	\end{enumerate}
\end{theorem}

Let us now turn to the proof of Theorem~\ref{thm:decay_rate_equivalent}.
First we describe how polynomial rates of decay of $C_0$-semigroups
can be transferred to resolvent growth on $\mathbb{C}_+$
in the Banach space setting.
The following result is obtained by a slight modification of the proof of
the implication (ii) $\Rightarrow$ (i) of \cite[Lemma~2.3]{Bounit2010}.
Although analyticity of $C_0$-semigroups
is assumed in \cite[Lemma~2.3]{Bounit2010}, 
the implication (ii) $\Rightarrow$ (i) there can be proved only under  a 
suitable boundedness condition. Recall that $1/p \coloneqq 0$ for $p =\infty$.
\begin{lemma}
	\label{lem:poly_decay_to_WC}
	Let $X$ and $Y$ be Banach spaces.
	Let $(T(t))_{t \geq 0}$ be a bounded $C_0$-semigroup on  $X$ with generator $A$, and
	let $C \in \mathcal{L}(D(A),Y)$ be such that
	$CT(t)$ extends to
	an operator (also denoted by $CT(t)$) in $\mathcal{L}(X,Y)$  for all $t >0$. If
	there exist $M>0$ and $1< p  \leq \infty$ such that 
	\[
	\|CT(t)\| \leq \frac{M}{t^{1/p}}
	\]
	for all $t >0$, 
	then $C$ satisfies the $p$-Weiss condition for $A$.
\end{lemma}

Next we 
provide
a converse result, i.e.,
transference from resolvents to semigroups, in the Hilbert space setting.
A key technique for the proof is contained in the proof of \cite[Theorem~4.7]{Batty2016}
and the paragraph after \cite[Remark~4.8]{Batty2016}, which considered the case $F(\xi )
\equiv K >0$ in the proposition below.
\begin{proposition}
	\label{prop:WC_to_poly_decay}
	Let $(T(t))_{t \geq 0}$ be a bounded $C_0$-semigroup on a Hilbert
	space $X$ with generator $A$.
	Let $C \in \mathcal{L}(D(A),X)$ be such that 
	$T(t)$ commute with $C$ for all $t \geq 0$.
	Assume that $F\colon (0,\infty) \to [0,\infty)$ satisfies
	\begin{equation}
	\label{eq:F_bound}
	\|CR(\lambda,A)\| \leq F(\re \lambda)
	\end{equation}
	for all $\lambda \in \mathbb{C}_+$.
	Then
	the operator $CT(t)$ extends to
	a bounded linear operator (also denoted by $CT(t)$) on $X$ for all $t >0$,
	and there exists $M >0$
	such that 
	\[
	\|CT(t)\| \leq M \frac{F(1/t)}{t}
	\] 
	for all $t >0$. 
\end{proposition}

\begin{proof}
	Let $x \in D(A)$ be given.
	Using Plancherel's theorem, we obtain
	\begin{align}
	\label{eq:tauTA}
	\int_0^{\infty}
	\|t e^{-\xi t } CT(t) x\|^2 dt
	&=
	\frac{1}{2\pi}
	\int_{-\infty}^{\infty}
	\|C R(\xi +i\eta,A)^2 x\|^2 d\eta
	\end{align}
	for all $\xi  >0$.
	By the resolvent estimate \eqref{eq:F_bound},
	\begin{align*}
	\|CR(\lambda,A)^2x\|^2 &\leq 
	\|CR(\lambda,A)\|^2 \,\|R(\lambda,A)x\|^2 \\
	&\leq 
	F(\re \lambda)^2\|R(\lambda,A)x\|^2
	\end{align*}
	for all $\lambda \in \mathbb{C}_+$.
	Therefore,
	\begin{align}
	\int_{-\infty}^{\infty}
	\|CR(\xi +i\eta,A)^2x\|^2 d\eta 
	\leq 
	F(\xi )^2
	\int_{-\infty}^{\infty}\|R(\xi +i\eta,A)x\|^2 d\eta
	\label{eq:Resol_a}
	\end{align}
	for all $\xi  >0$. Using Plancherel's theorem again, we have 
	\begin{align}
	\int_{-\infty}^{\infty}\|R(\xi +i\eta,A)x\|^2 d\eta
	&=
	2\pi \int_0^{\infty}  \|e^{-\xi t}T(t)x\|^2 dt \notag \\
	&\leq \frac{\pi c^2}{\xi } \|x\|^2
	\label{eq:R_x_bound}
	\end{align}
	for all $\xi  >0$,
	where $c \coloneqq \sup_{t \geq 0}\|T(t)\|$. Combining \eqref{eq:tauTA}--\eqref{eq:R_x_bound},
	we obtain
	\begin{equation}
	\label{eq:teCT_bound}
	\int_0^{\infty}
	\|t  e^{-\xi t } CT(t )x\|^2 dt
	\leq \frac{c^2  F(\xi )^2}{2\xi } \|x\|^2
	\end{equation}
	for all $\xi  >0$.
	Since
	\[
	e^{-2} \|t   CT(t ) x\|^2 \leq 
	\|t e^{-t/\tau}   CT(t ) x\|^2 
	\]
	for $0\leq t \leq \tau$,
	the estimate \eqref{eq:teCT_bound} with 
	$\xi  \coloneqq 1/\tau$
	yields
	\begin{align*}
	\int_0^{\tau}
	\|t CT(t )x\|^2 dt
	&\leq 
	e^2\int_0^{\tau}
	\|t e^{-t/\tau}   CT(t ) x\|^2 dt \\
	&\leq 
	\frac{e^2c^2 }{2} \tau F(1/\tau)^2\|x\|^2 
	\end{align*} 
	for all $\tau >0$.

	Fix $\tau >0$.
	Since
	\[
	CT(\tau )x = 
	\frac{2}{\tau^2}
	\int_0^{\tau} t T(\tau-t) CT(t) x dt,
	\]
	the Cauchy-Schwarz inequality implies that  for all $y \in X$,
	\begin{align*}
	|\langle CT(\tau) x,y  \rangle| 
	&=
	\frac{2}{\tau^{2}}
	\left|
	\left\langle
	\int_0^{\tau} t  T(\tau-t) C T(t ) x dt, y
	\right\rangle
	\right| \\
	& =
	\frac{2}{\tau^{2}}
	\left|
	\int_0^{\tau}
	\left\langle
	t  CT(t) x , T(\tau-t)^*y
	\right \rangle dt
	\right| \\
	&\leq 
	\frac{2}{\tau^{2}}
	\sqrt{
		\int_0^{\tau}   \|t CT(t) x\|^2 dt
	}\,
	\sqrt{
		\int_0^{\tau}  \|T(\tau-t)^*y\|^2 dt
	}  \\
	& \leq 
	\sqrt{2}e c^2 \frac{F(1/\tau)}{\tau} \|x\|\, \|y\|,
	\end{align*}
	where $T(t)^*$ is
	the Hilbert space adjoint of $T(t)$ for $t \geq 0$.
	Since $D(A)$ is dense in $X$, it follows that 
	$CT(\tau)$ extends to
	a bounded linear operator on $X$ with norm at most 
	$\sqrt{2}ec^2 F(1/\tau) /\tau$.
\end{proof}

Let $K >0$ and $1 \leq p <\infty$.
If the function $F$ in Proposition~\ref{prop:WC_to_poly_decay} is defined by
\begin{equation}
\label{eq:F_Weiss_cond}
F(\xi ) \coloneqq \frac{K}{\xi ^{1-1/p}}
\end{equation}
for $\xi  >0$, then
the resolvent estimate \eqref{eq:F_bound} coincides with that in 
the $p$-Weiss condition on $C$.
Since
\[
\frac{F(1/t)}{t} = \frac{K}{t^{1/p}},
\]
Proposition~\ref{prop:WC_to_poly_decay} shows that
if $C$ satisfies the $p$-Weiss condition, then
\[
\|CT(t)\| \leq \frac{M}{t^{1/p}}
\] 
for all $t>0$ and  some $M>0$.

\begin{proof}[Proof of Theorem~\Rref{thm:decay_rate_equivalent}]
	If the statement (i) holds, then
	we see from \cite[Theorem~1.1]{Batty2008}  that $i\mathbb{R} \subset \varrho(A)$,
	and therefore the implication (i) $\Rightarrow $ (ii) is true
	in the case $p=1$ by Theorem~\ref{thm:uniform_bounded}.
	By Proposition~\ref{prop:interpolation}, 
	the statement (i) is true if and only if
	$\|T(t)(-A)^{-\alpha/p}\| = O(t^{-1/p})$ as $t \to \infty$. 
	Hence,
	the implication (i) $\Rightarrow $ (ii)
	in the case $1 < p < \infty$ 
	and the implication (ii) $\Rightarrow $ (i) follow immediately
	from 
	Lemma~\ref{lem:poly_decay_to_WC}  and
	Proposition~\ref{prop:WC_to_poly_decay} with $C \coloneqq (-A)^{-\alpha/p}$,
	respectively.
\end{proof}

\begin{proof}[Proof of Theorem~\Rref{thm:singularity_zero}]
	For each $k \in \{1,2 \}$, if the statement (i) holds, then
	$\sigma(A) \cap i\mathbb{R} = \{ 0\}$
	by \cite[Corollary~6.2]{Batty2016}.
	Hence the implication (i) $\Rightarrow $ (ii) in the case $p=1$ 
	is true for  $k=1$ by
	Corollary~7.5 and Theorem~7.6 of \cite{Batty2016}
	and for $k=2$ by
	Theorem~8.4 and its proof of \cite{Batty2016}.
	Since Proposition~\ref{prop:interpolation} shows that
	the statement (i) is equivalent to
	\[
	\|T(t) (-A)^{\alpha/p}(I-A)^{-k\alpha/p}\| = 
	O\left(\frac{1}{t^{1/p}}
	\right)\qquad \text{as $t \to \infty$}
	\]
	for each $k \in \{1,2 \}$,
	the rest follows from
	Lemma~\ref{lem:poly_decay_to_WC}
	and Proposition~\ref{prop:WC_to_poly_decay} with 
	\[
	C \coloneqq (-A)^{\alpha/p}(I-A)^{-k\alpha/p}
	\]
	as in the proof of Theorem~\ref{thm:decay_rate_equivalent}.
\end{proof}

\subsection{Decay of Banach space semigroups and infinite-time admissibility}
Next, we study the relation
between polynomial decay and $L^p$-infinite-time admissibility.
We have seen in
Proposition~\ref{prop:WC_to_poly_decay} that
the $p$-Weiss condition on $C \in \mathcal{L}(D(A),X)$
implies $\|CT(t)\| = O(t^{-1/p})$
in the Hilbert space setting if $T(t)$ commutes with $C$ for all $t \geq 0$.
We obtain a similar result from
the $L^p$-infinite-time admissibility of $C$, which is a stronger 
property than the $p$-Weiss condition,
in the Banach space setting. 
This follows from the technique used 
in the proof given in \cite{Pritchard1981} for Datko's theorem.
Parts of the next result was proved in a slightly less general form (and with a slightly different proof)
in \cite[Theorem~2.5]{Zwart2012}.
\begin{proposition}
	\label{prop:p_ad_C}
	Let $(T(t))_{t \geq 0}$ be a $C_0$-semigroup on a Banach 
	space $X$ with generator $A$, and let
	$C \in \mathcal{L}(D(A),X)$ be such that $T(t)$ commutes with $C$ for all $t \geq 0$.
	Let $1\leq p < \infty$ and
	assume that $C$ is $L^p$-infinite-time admissible for $A$. Then
	the following statements are true:
	\begin{enumerate}[\em \alph{enumi})]
		\item
		For all $t >0$,
		$CT(t)$ extends to
		a bounded linear operator (also denoted by $CT(t)$) on $X$.
		Moreover,
		for all $t_0 >0$, there exists $M >0$
		such that 
		\[
		\|CT(t)\| \leq M
		\] 
		for all $t \geq t_0$. 
		\item
		If $(T(t))_{t \geq 0}$ is a bounded $C_0$-semigroup,
		then there exists $M >0$
		such that 
		\[
		\|CT(t)\| \leq \frac{M}{t^{1/p}}
		\] 
		for all $t >0$. 
	\end{enumerate}
	
\end{proposition}
\begin{proof}
	By the $L^p$-infinite-time admissibility of $C$, there exists
	a constant $M_1 >0$ such that 
	\[
	\int_0^\infty
	\|C T(\tau)x\|^p d\tau \leq M_1^p \|x\|^p
	\] 
	for all $x \in D(A)$.
	
	a) There exists $M_2 \geq 1$ and $w >0$ such that 
	\[
	\|T(t)\| \leq M_2 e^{w t}
	\]
	for all $t \geq 0$.
	We have
	\begin{align*}
	\int^t_0 e^{-pw \tau } d\tau \|CT(t)x\|^p
	&\leq
	\int^t_0 e^{-pw \tau } \|T(\tau) CT(t-\tau)x\|^p d\tau \\
	&\leq 
	M_2^p\int^t_0 \|CT(t-\tau)x\|^p d\tau \\
	&\leq M_1^p M_2^p \|x\|^p
	\end{align*}
	for all $t \geq 0$ and $x \in D(A)$. 
	Therefore,
	\[
	\|CT(t)x\| \leq \left(
	\frac{pw}{1-e^{-pwt}}
	\right)^{1/p}M_1M_2 \|x\|
	\]
	for all $t >0$ and $x \in D(A)$.
	For all $t >0$, the density of $D(A)$ implies that
	$CT(t)$ extends to
	a bounded linear operator on $X$, 
	again denoted by $CT(t)$. We also have
	\[
	\|CT(t)\| \leq \left(
	\frac{pw}{1-e^{-pwt_0}}
	\right)^{1/p}M_1M_2
	\]
	for all $t \geq  t_0 >0$.
	
	b)
	By assumption, we have
	\[
	c \coloneqq \sup_{t \geq 0}
	\|T(t)\|< \infty.
	\]
	For all $t \geq 0$ and $x \in D(A)$,
	\begin{align*}
	t\|CT(t) x\|^p
	&=
	\int^t_0 
	\|T(t-\tau)C T(\tau) x\|^p d\tau \\
	&\leq 
	c^p\int^\infty_0 
	\|C T(\tau)  x\|^p d\tau \\
	&\leq c^p M_1^p \|x\|^p.
	\end{align*}
	From the density of $D(A)$, it follows that 
	$CT(t)$ satisfies
	\[
	\|CT(t)\| \leq \frac{cM_1}{t^{1/p}}
	\]
	for all $t > 0$.
\end{proof}

We obtain the following proposition from a simple argument.
\begin{proposition}
	\label{prop:resolvent_admissible_decay}
	Let $(T(t))_{t \geq 0}$ be a bounded $C_0$-semigroup on a Banach 
	space $X$ with generator $A$ such that  $0 \in \varrho(A)$. 
	For fixed $\alpha >0$ and $1\leq p < \infty$,
	the following statements are true:
	\begin{enumerate}[\em \alph{enumi})]
		\item
		If $(-A)^{-\alpha/p}$ is $L^p$-infinite-time admissible for $A$, then
		$\|T(t)A^{-1}\| = 
		O(t^{-1/\alpha})
		$ as $t \to \infty$.
		\item
		If 
		$\|T(t) A^{-1}\| = 
		O(t^{-1/\alpha})
		$ as $t \to \infty$,
		then $(- A)^{-\alpha/p-\varepsilon}$ 
		with $\varepsilon > 0$ is $L^p$-infinite-time admissible  for $A$.
	\end{enumerate}
\end{proposition}
\begin{proof}
	a) This follows immediately from Proposition~\ref{prop:p_ad_C}.b) with
	$C \coloneqq (-A)^{-\alpha/p}$ and
	Proposition~\ref{prop:interpolation}.
	
	b) Let $\beta > \alpha /p$. By Proposition~\ref{prop:interpolation}, 
	there exists a constant $M>0$
	such that
	\[
	\|T(t) (-A)^{-\beta}\|  \leq \frac{M}{(1+t)^{\beta/\alpha}}
	\]
	for all $t \geq 0$. Since $p\beta/\alpha > 1$, it follows that 
	for all $x \in X$,
	\begin{align*}
	\int_0^{\infty} 
	\|(-A)^{-\beta} T(t)x \|^p dt
	&\leq
	\int_0^{\infty} 
	\frac{M^p \|x\|^p }{(1+t)^{p\beta/\alpha}} dt \\
	&\leq 
	\frac{M^p \|x\|^p}{p\beta/\alpha - 1}.
	\end{align*}
	Hence $(-A)^{-\beta}$ is infinite-time admissible for $A$.
\end{proof}

It remains open whether $\|T(t)A^{-1}\| = O(t^{-1/\alpha})$ implies the $L^p$-infinite-time
admissibility of $(-A)^{\alpha/p}$ for all bounded $C_0$-semigroups 
$(T(t))_{t \geq 0}$ on
Banach spaces. However,
it is true
for multiplication $C_0$-semigroups on $L^q$-spaces with $1 \leq q \leq p < \infty$.
\begin{theorem}
	\label{thm:Lp_ad}
	Let $\mu$
	be a $\sigma$-finite regular Borel measure on a locally compact Hausdorff space  $\Omega$.
	Let
	$\phi \colon \Omega \to \mathbb{C}$ be measurable with essential range 
	$\phi_{\mathrm{ess}}(\Omega)$
	in $\mathbb{C}_-$.
	Assume that $A$ is the multiplication operator induced by $\phi$ on $L^q(\Omega,\mu)$ for a fixed $1 \leq q < \infty$, i.e., $Af = \phi f$
	with domain $D(A) \coloneqq \{f \in L^q(\Omega,\mu): \phi f \in L^q(\Omega,\mu)  \}$,
	and let $(T(t))_{t\geq 0}$ be the $C_0$-semigroup on $L^q(\Omega,\mu)$ generated by $A$.  Then
	the following statements are equivalent for fixed $\alpha>0$ and $p \in [q,\infty)$:
	\begin{enumerate}[\em (\roman{enumi})]
		\item $\|T(t)A^{-1}\| = O(t^{-1/\alpha})$ as $t \to \infty$.
		\item $(-A)^{-\alpha/p}$ is $L^p$-infinite-time admissible for $A$.
	\end{enumerate}
\end{theorem}

\begin{proof}
	We first note that $\phi(z) \in \phi_{\mathrm{ess}}(\Omega)$
	for almost all $z \in \Omega$; see, e.g., Exercise~19 of Chapter~VII of \cite{Lang1993}.
	Since $\phi_{\mathrm{ess}}(\Omega) \subset \mathbb{C}_-$ by assumption,
	it follows that $(T(t))_{t\geq 0}$ is a bounded $C_0$-semigroup.
	The implication (ii) $\Rightarrow$ (i) has been shown in Proposition~\ref{prop:resolvent_admissible_decay}. Therefore,
	we here prove the implication (i) $\Rightarrow$ (ii).

	For all $t \geq 0$ and $f \in L^q(\Omega,\mu)$,
	\begin{align}
	\|T(t) (-A)^{-\alpha/p} f\|^p &=
	\|e^{t \phi} (-\phi)^{-\alpha/p} f\|^p \notag \\
	&=
	\left(
	\int_{\Omega}
	\frac{e^{qt \re \phi(z)}}{|\phi(z)|^{q \alpha /p}}
	|f(z)|^q d\mu(z)
	\right)^{p/q}.
	\label{eq:Lp1}
	\end{align}
	Since $p \geq q$, 
	Jensen's inequality (see, e.g., Theorem~3.3 of 
	\cite{Rudin1987}) implies that
	\begin{align}
	\left(
	\int_{\Omega}
	\frac{e^{qt \re \phi(z)}}{|\phi(z)|^{q \alpha /p}}
	\frac{|f(z)|^q}{\|f\|^q} d\mu(z)
	\right)^{p/q} &\leq 
	\int_{\Omega}
	\left(
	\frac{e^{qt \re \phi(z)}}{|\phi(z)|^{q \alpha /p}}
	\right)^{p/q}
	\frac{|f(z)|^q}{\|f\|^q} d\mu(z) \notag \\
	&=
	\int_{\Omega}
	\frac{e^{pt \re \phi(z)}}{|\phi(z)|^{\alpha }}
	\frac{|f(z)|^q}{\|f\|^q} d\mu(z).
	\label{eq:Lp2}
	\end{align}

	By the assumption $\phi_{\mathrm{ess}}(\Omega) \subset \mathbb{C}_-$, we obtain
	\begin{align}
	\label{eq:int_multiplication}
	\int^{\infty}_0 
	\frac{e^{pt \re \phi(z)}}{|\phi(z)|^{\alpha }} dt
	&=
	\frac{1}{p  |\re \phi(z)|\, |\phi(z)|^{\alpha}}
	\end{align}
	for almost all $z \in \Omega$.
	Since $\|T(t)A^{-\alpha}\| = O(t^{-1})$ as $t \to \infty$ by
	Proposition~\ref{prop:interpolation},
	it follows from \cite[Proposition~4.2]{Batkai2006} that
	there exist $\Upsilon,\kappa >0$ such that 
	\[
	\frac{1}{|\re \lambda|} \leq \Upsilon  |\im \lambda|^{\alpha}
	\]
	for all $\lambda \in \sigma(A)$ 
	satisfying $|\re \lambda| \leq \kappa$.
	Therefore, 
	if $\lambda \in \sigma(A)$ satisfies $|\re \lambda| \leq \kappa$, then
	\[
	\frac{1 }{ |\re \lambda|\, |\lambda|^\alpha } \leq 
	\frac{\Upsilon  |\im \lambda|^{\alpha} }{|\lambda|^\alpha} \leq \Upsilon .
	\]
	If $\lambda \in \sigma(A)$ satisfies  $|\re \lambda| > \kappa$, then
	\[
	\frac{1 }{ |\re \lambda|\, |\lambda|^\alpha } \leq \frac{1}{\kappa^{1+\alpha}}.
	\]
	Define
	\[
	M \coloneqq \max\left\{
	\Upsilon,\, \frac{1}{\kappa^{1+\alpha}}
	\right\}.
	\]
	Since 
	$\phi_{\mathrm{ess}}(\Omega) = \sigma(A)$ by \cite[Proposition~I.4.10]{Engel2000},
	we have that 
	$\phi(z) \in \sigma(A)$ for almost all $z \in \Omega$.
	Hence
	\begin{equation*}
	\frac{1}{|\re \phi(z)|\, |\phi(z)|^{\alpha}} \leq M
	\end{equation*}
	for almost all $z \in \Omega$.
	This estimate and \eqref{eq:int_multiplication} yield
	\begin{equation}
	\label{eq:Lp3}
	\int_{\Omega} \int_0^{\infty}
	\frac{e^{pt \re \phi(z)}}{|\phi(z)|^{\alpha }}
	\frac{|f(z)|^q}{\|f\|^q} dt d\mu(z)
	\leq 
	\frac{M}{p}
	\int_{\Omega}\frac{|f(z)|^q}{\|f\|^q} d\mu(z)
	=\frac{M}{p}.
	\end{equation}
	
	Combining Fubini's theorem with \eqref{eq:Lp1}, \eqref{eq:Lp2}, and 
	\eqref{eq:Lp3}, we derive
	\begin{align*}
	\int^\infty_0 \| T(t)(-A)^{-\alpha/p} f \|^p dt 
	&\leq \frac{M}{p} \|f\|^p.
	\end{align*}
	Thus, $(-A)^{-\alpha/p}$ is $L^p$-infinite-time admissible for $A$.
\end{proof}

\subsection{Example}
We present a simple example, for which 
the estimates for $\|T(t)A^{-1}\|$ in the statement (i) of 
Theorems~\ref{thm:decay_rate_equivalent}~and~\ref{thm:Lp_ad} 
are optimal in the sense that 
the rate of decay cannot be improved.

Let $X = \ell^2(\mathbb{N})$ and 
define an operator $A\colon D(A) \subset X \to X$ by
\[
Ax \coloneqq 
\left(
\left(
-\frac{1}{n} + i n
\right) x_n
\right)_{n \in \mathbb{N}}
\]
for 
\[
x = (x_n)_{n \in \mathbb{N}} \in D(A) \coloneqq
\{
\zeta = (\zeta_n)_{n \in \mathbb{N}} \in \ell^2(\mathbb{N}) : (n\zeta_n)\in \ell^2(\mathbb{N})
\}.
\]
The semigroup $(T(t))_{t \geq 0}$ generated by $A$ is given by
\[
T(t)x = \left(
e^{\left(
	-\frac{1}{n} + i n
	\right)t} x_n
\right)_{n \in \mathbb{N}}
\]
for all $t \geq 0$ and $x = (x_n)_{n \in \mathbb{N}} \in \ell^2(\mathbb{N})$.
Let $\alpha >0$.

\subsubsection{Weiss condition}
For all $\lambda \in \mathbb{C}_+$,
\begin{equation}
\label{eq:weiss_ex1}
\sqrt{\re \lambda } \, \|(-A)^{-\alpha} R(\lambda,A)\| = 
\sup_{n \in \mathbb{N}} \frac{\sqrt{\re \lambda}}{
	|1/n-in|^\alpha \, |\lambda + 1/n-in| }.
\end{equation}
Moreover,
\begin{equation}
\label{eq:weiss_ex_bound}
\frac{\sqrt{\re \lambda}}{
	|1/n-in|^\alpha\, |\lambda + 1/n-in| }
\leq 
\frac{\sqrt{\re \lambda}}{n^{\alpha} (\re \lambda + 1/n )}
\end{equation}
for all $\lambda \in \mathbb{C}_+$ and $n \in \mathbb{N}$.
Fix $c > 0$, and
define
\[
f(\xi ) \coloneqq \frac{\sqrt{\xi }}{\xi +c}
\]
for $\xi >0$. Then
\[
f'(\xi ) = \frac{c-\xi }{2 \sqrt{\xi }\, (\xi +c)^2}.
\]
Therefore, $f(\xi) \leq f(c) = 1/(2\sqrt{c})$ for all $\xi >0$. Combining this with the estimate 
\eqref{eq:weiss_ex_bound}, we obtain
\begin{equation}
\label{eq:weiss_ex2}
\frac{\sqrt{\re \lambda}}{
	|1/n-in|^\alpha\, |\lambda + 1/n-in| } \leq 
\frac{n^{1/2-\alpha}}{2}
\end{equation}
for all $\lambda \in \mathbb{C}_+$ and $n \in \mathbb{N}$.
On the other hand,
since
\[
|1/n-in|^\alpha \leq 2^{\alpha / 2}n^\alpha,
\]
it follows that 
for each $n \in \mathbb{N}$, the case
$\re \lambda = 1/n$ and $\im \lambda = n$ gives
\begin{equation}
\label{eq:weiss_ex3}
\frac{\sqrt{\re \lambda}}{
	|1/n-in|^\alpha \, |\lambda + 1/n-in| }
\geq \frac{n^{1/2-\alpha}}{2^{1+\alpha/2}}.
\end{equation}
Substituting
the estimates \eqref{eq:weiss_ex2} and \eqref{eq:weiss_ex3} into \eqref{eq:weiss_ex1},
we see that
$(-A)^{-\alpha}$ satisfies the $2$-Weiss condition for $A$ if and only if
$\alpha \geq 1/2$.

\subsubsection{Infinite-time admissibility}
We directly obtain a similar result on $L^2$-infinite-time admissibility.
Indeed,
using the monotone convergence theorem,
we have that 
for all 
$x = (x_n)_{n \in \mathbb{N}} \in \ell^2(\mathbb{N})$,
\begin{align*}
\int^{\infty}_0
\|(-A)^{-\alpha}T(t) x\|^2 dt &=
\int^{\infty}_0
\sum_{n=1}^\infty
\left|
\left(
-\frac{1}{n} + in
\right)^{-\alpha}
e^{
	\left(
	-\frac{1}{n} + in
	\right)t
}x_n
\right|^2 dt \\
&=
\frac{1}{2}
\sum_{n=1}^\infty
\frac{n^{1+2\alpha}}{(1+n^4)^{\alpha}} |x_n|^2.
\end{align*}
Therefore,
$(-A)^{-\alpha}$ is $L^2$-infinite-time admissible for $A$ if and only if
$\alpha \geq 1/2$.

\subsubsection{Polynomial decay rate}
Fix $c>0$. Then
$g(t) \coloneqq t{e^{-t/c}} \leq g(c) = ce^{-1}$ for $t \geq 0$. From this inequality, 
we obtain 
\[
t\|T(t)A^{-1}\| = 
\sup_{n \in \mathbb{N}} \frac{nte^{-\frac{t}{n}}}{\sqrt{1+n^4}} 
\leq \sup_{n \in \mathbb{N}} \frac{e^{-1}n^2}{\sqrt{1+n^4}} \leq e^{-1}
\]
for all $t \geq 0$.
Therefore, $\|T(t)A^{-1}\| = O(1/t)$ as $t \to \infty$. 
We also have 
\[
n\|T(n)A^{-1}\| \geq \frac{e^{-1} n^2}{\sqrt{1+n^4}} 
\] 
for all $n \in \mathbb{N}$, and hence
\[
\liminf_{t \to \infty} t\|T(t)A^{-1}\| \geq e^{-1}.
\]
This demonstrates the optimality of the rates of decay when the statement
(ii)  of Theorem~\ref{thm:decay_rate_equivalent}~or~\ref{thm:Lp_ad}
holds for $p=2$.

\section{Non-polynomial decay of Hilbert space semigroups}
\label{sec:non_polynomial}
Let $X$ and $Y$ be Hilbert spaces, and
let $(T(t))_{t \geq 0}$ be a bounded $C_0$-semigroup on $X$
with generator $A$.
By modifying \cite[Theorem~4.3]{Zwart2005} slightly, we have that
if $C \in \mathcal{L}(D(A),Y)$ satisfies
$\sup_{\eta \in \mathbb{R}} \|CR(1+i\eta,A)\| < \infty$ and
\begin{equation}
\label{eq:resol_cond_poly_log}
\|CR(\lambda,A)\| \leq \frac{K}{\sqrt{\re \lambda}\, |\log \re \lambda |}
\end{equation}
for all $\lambda \in \mathbb{C}_+$ with $\re \lambda \not= 1$ and some $K>0$, then
$C$
is $L^2$-infinite-time admissible for $A$. 
We see from Proposition~\ref{prop:WC_to_poly_decay} 
that
if $Y = X$ and if $T(t)$ commutes with $C$ for all $t \geq 0$, then 
the estimate \eqref{eq:resol_cond_poly_log} 
implies
\begin{equation}
\label{eq:orbit_poly_log}
\|CT(t)\| \leq \frac{M}{\sqrt{t}\, |\log t|}
\end{equation}
for all $t>0$ with $t \not= 1$ and some $M>0$.
Note that the estimate \eqref{eq:resol_cond_poly_log}
contains the condition on
the decay rate of 
$\|CR(\lambda,A)\|$ as
$\re \lambda \to \infty$
in addition to the growth rate of $\|CR(\lambda,A)\|$ as
$\re \lambda \to 0+$, 
since $C$ may not belong to $\mathcal{L}(X,Y)$.
Similarly, the estimate \eqref{eq:orbit_poly_log}  includes the condition on
the growth rate of 
$\|CT(t)\|$
as $t \to 0+$.
In this section, we assume that $C \in \mathcal{L}(X,Y)$, and then show how the decay
estimate $\|CT(t)\| = O(t^{-\beta} (\log t)^{-\gamma})$ as $t\to \infty$ can be transferred to
the growth estimate of $\|CR(\lambda,A)\|$ as $\re \lambda \to 0+$.

\begin{proposition}
	\label{prop:t_logt}
	Let $X$ and $Y$ be Banach spaces, and
	let $(T(t))_{t \geq 0}$ be a bounded $C_0$-semigroup on $X$
	with generator $A$. 
	For fixed $0 \leq \beta < 1$ and $\gamma \geq 0$, 
	the following statements hold.
	\begin{enumerate}[\em \alph{enumi})]
		\item
		If
		$C \in \mathcal{L}(X,Y)$
		satisfies
		\[
		\|CT(t)\| = O \left(
		\frac{1}{t^\beta (\log t)^\gamma}
		\right)
		\]
		as $t \to \infty$,
		then there exists $K >0$ such that
		\begin{equation*}
		\|CR(\lambda,A)\| \leq
		\frac{K}{(\re \lambda)^{1-\beta} |\log \re \lambda|^\gamma}
		\end{equation*}
		for all $\lambda \in \mathbb{C}_+$ with $\re \lambda < 1$.
		\item
		If
		$C \in \mathcal{L}(X,Y)$
		satisfies
		\[
		\|CT(t)\| = O \left(
		\frac{1}{t (\log t)^\gamma}
		\right)
		\]
		as $t \to \infty$,
		then, for all $\xi_0 \in (0,e^{-1})$, 
		there exists $K >0$ such that
		\[
		\|CR(\lambda,A)\| \leq K F_\gamma(\re \lambda)
		\]
		for all $\lambda \in \mathbb{C}_+$ with $\re \lambda < \xi_0$, where
		\begin{equation}
		\label{eq:F_gamma_def}
		F_\gamma(\xi ) \coloneqq 
		\begin{cases}
		|\log \xi |^{1-\gamma}, & 0 \leq \gamma < 1, \\
		\log|\log \xi |, & \gamma = 1,\\
		1, & \gamma > 1
		\end{cases}
		\end{equation}
		for $0 < \xi  < e^{-1}$.
	\end{enumerate}
\end{proposition}

The following lemma will be useful in the proof of Proposition~\ref{prop:t_logt}.
\begin{lemma}
	\label{lem:integral_estimate}
	For fixed $0 \leq \beta < 1$ and $\gamma \geq 0$, the following statements hold:
	\begin{enumerate}[\em \alph{enumi})]
		\item
		For all $t_0 > e^{\gamma/(1-\beta)}$, there exists $M>0$ such that
		\[
		\int^{\infty}_{t_0}
		\frac{e^{- \xi t }}{t^\beta (\log t)^\gamma} dt \leq  \frac{M}{\xi ^{1-\beta}|\log \xi |^\gamma}
		\]
		for all $0 < \xi  < 1/t_0$.
		\item
		For all $t_0 > e$, there exists $M>0$ such that
		\[
		\int^{\infty}_{t_0}
		\frac{e^{- \xi t }}{t (\log t)^\gamma} dt \leq  MF_\gamma(\xi )
		\]
		for all $0 <\xi  < 1/t_0$, where
		$F_\gamma$ is defined by \eqref{eq:F_gamma_def}.
	\end{enumerate}
\end{lemma}
\begin{proof}
	a)
	Let $t_0 > e^{\gamma/(1-\beta)}$ and
	$0< \xi  < 1/t_0$.
	Then
	\begin{align}
	\int^{\infty}_{1/\xi }
	\frac{e^{-\xi t }}{t^\beta (\log t)^\gamma} dt \leq 
	\frac{1}{(1/\xi )^\beta |\log \xi |^\gamma} \int^{\infty}_{1/\xi } e^{-\xi t }dt 
	= \frac{e^{-1}}{\xi ^{1-\beta}|\log \xi |^\gamma}.
	\label{eq:int_est_larger}
	\end{align}
	We also have 
	\begin{align}
	\label{eq:int_est_smaller_pre}
	\int^{1/\xi }_{t_0}
	\frac{e^{-\xi t }}{t^\beta (\log t)^\gamma} dt 
	\leq 
	\int^{1/\xi }_{t_0}
	\frac{1}{t^\beta (\log t)^\gamma} dt.
	\end{align}
	Integration by parts gives
	\begin{align}
	\int^{1/\xi }_{t_0}
	\frac{1}{t^\beta (\log t)^\gamma} dt  
	&=
	\frac{(1/\xi )^{1-\beta} }{(1-\beta) |\log \xi |^\gamma }
	-
	\frac{t_0^{1-\beta} }{(1-\beta) (\log t_0)^\gamma }
	+
	\frac{\gamma}{1-\beta}\int^{1/\xi }_{t_0}
	\frac{1}{t^\beta (\log t)^{\gamma+1}} dt   \notag \\
	& \leq 
	\frac{(1/\xi )^{1-\beta} }{(1-\beta) |\log \xi |^\gamma }
	+
	\frac{\gamma}{(1-\beta)\log t_0} \int^{1/\xi }_{t_0}
	\frac{1}{t^\beta (\log t)^{\gamma}} dt .
	\label{eq:int_estimate_exp_poly}
	\end{align}
	Since $t_0 > e^{\gamma/(1-\beta)}$ is equivalent to
	\[
	\frac{\gamma}{(1-\beta) \log t_0 } < 1,
	\]
	the estimates \eqref{eq:int_est_smaller_pre}
	and  \eqref{eq:int_estimate_exp_poly} yield
	\begin{equation}
	\label{eq:int_est_smaller}
	\int^{1/\xi }_{t_0}
	\frac{e^{-\xi t }}{t^\beta (\log t)^\gamma} dt  \leq 
	\frac{M_0}{\xi ^{1-\beta} |\log \xi |^\gamma },
	\end{equation}
	where
	\[
	M_0 \coloneqq 
	\frac{1}{1-\beta}
	\left(
	1 - 
	\frac{\gamma}{(1-\beta) \log t_0 }
	\right).
	\]
	The assertion with $M \coloneqq M_0 + e^{-1}$ then follows from 
	the estimates \eqref{eq:int_est_larger} and 
	\eqref{eq:int_est_smaller}, noting that $M$ is independent of $\xi$.
	
	b) 
	Let $t_0 > e$ and
	$0< \xi  < 1/t_0$.
	Since
	\[
	\int^{1/\xi }_{t_0}
	\frac{e^{-\xi t }}{t(\log t)^\gamma} dt  
	\leq \int^{1/\xi }_{t_0}
	\frac{1}{t(\log t)^\gamma} dt  
	=
	\int_{\log t_0}^{\log(1/\xi )} 
	\frac{1}{\tau^\gamma} d\tau,
	\]
	we obtain
	\begin{align*}
	\int^{1/\xi }_{t_0}
	\frac{e^{-\xi t }}{t(\log t)^\gamma} dt  
	& \leq
	\begin{cases}
	\dfrac{|\log \xi |^{1-\gamma} - 
		(\log t_0)^{1-\gamma}}{1-\gamma}, & \gamma \not=1, \\
	\log|\log \xi | - \log (\log t_0),
	& \gamma = 1.
	\end{cases}
	\end{align*}
	Define
	\[
	G_\gamma(\xi ) \coloneqq 
	\begin{cases}
	\dfrac{|\log \xi |^{1-\gamma}}{1-\gamma},  & 0\leq \gamma <1, \\
	\log|\log \xi |,
	& \gamma = 1, \vspace{3pt}\\
	\dfrac{(\log t_0)^{1-\gamma}}{\gamma-1},
	&\gamma > 1.
	\end{cases}
	\]
	Then, instead of the estimate \eqref{eq:int_est_smaller}, we have 
	\begin{equation*}
	\int^{1/\xi }_{t_0}
	\frac{e^{-\xi t }}{t (\log t)^\gamma} dt  \leq 
	G_\gamma(\xi ).
	\end{equation*}
	On the other hand, estimating as in \eqref{eq:int_est_larger} yields
	\begin{equation*}
	\int^{\infty}_{1/\xi }
	\frac{e^{-\xi t }}{t (\log t)^\gamma} dt \leq 
	\frac{e^{-1}}{|\log \xi |^\gamma} < \frac{e^{-1}}{(\log t_0)^\gamma}.
	\end{equation*}
	Since $t_0 > e$, the function
	$F_\gamma$ defined by \eqref{eq:F_gamma_def} satisfies
	$\inf_{0< \xi<1/t_0} F_{\gamma}(\xi) >0 $ for all $\gamma \geq 0$.
	Thus, the desired estimate is obtained for
	some $M>0$ depending only on $t_0$ and $\gamma$. 
\end{proof}

\begin{proof}[Proof of Proposition~\Rref{prop:t_logt}]
	a)
	By assumption, there exist $M_0>0$ and $t_0 > e^{\gamma/(1-\beta)}$ such that
	\[
	\|CT(t)\| \leq
	\frac{M_0}{t^\beta (\log t)^\gamma}
	\]
	for all $t \geq t_0$. 
	For all $x \in X$ and $\lambda \in \mathbb{C}_+$, 
	\[
	\|CR(\lambda,A)x\| \leq \int^{\infty}_0 e^{-t \re \lambda }\|CT(t)x\|dt.
	\]
	Define $c \coloneqq \sup_{t \geq 0} \|CT(t)\|$.
	For all $x \in X$ and $\lambda \in \mathbb{C}_+$,
	\begin{align*}
	\int^{\infty}_0 e^{-t \re \lambda }\|CT(t)x\|dt 
	&\leq 
	\int^{t_0}_0 c \|x\| dt + 
	\int^{\infty}_{t_0}
	e^{-t \re \lambda} \frac{M_0 \|x\|}{t^\beta (\log t)^\gamma} dt \\
	&=\left(
	ct_0 + 
	M_0
	\int^{\infty}_{t_0}
	\frac{e^{-t \re \lambda} }{t^\beta (\log t)^\gamma} dt\right) \|x\| .
	\end{align*}
	Hence
	by Lemma~\ref{lem:integral_estimate}.a),
	there exists $K_0 >0$ such that
	\[
	\|CR(\lambda,A)\| \leq
	\frac{K_0}{(\re \lambda)^{1-\beta} |\log \re \lambda|^\gamma}
	\]
	for all $\lambda \in \mathbb{C}_+$ with $\re \lambda < 1/t_0$.
	Combining this with the Hille-Yosida theorem,
	we obtain the desired estimate
	for all $\lambda \in \mathbb{C}_+$ with $\re \lambda < 1$ and some
	$K >0$.
	
	b)
	By the same argument as above,
	we have that 
	the statement b) is true, using 
	Lemma~\ref{lem:integral_estimate}.b) 
	instead of Lemma~\ref{lem:integral_estimate}.a).
\end{proof}

Let $(T(t))_{t \geq 0}$ be
a bounded $C_0$-semigroup 
on a Hilbert space with
generator $A$ such that $i\mathbb{R} \subset \varrho(A)$, and let $\alpha >0$.
From Proposition~\ref{prop:interpolation} and 
Theorem~\ref{thm:uniform_bounded},
we already know  that $\|T(t)(-A)^{-\alpha}\| = O(1/t)$ as $t\to \infty$ if and only if
$
\sup_{\lambda \in \mathbb{C}_+}\|R(\lambda,A)(-A)^{-\alpha}\| < \infty.
$
However, Proposition~\ref{prop:t_logt}.b) with $C \coloneqq (-A)^{-\alpha}$ and $\gamma = 0$
only shows that if $\|T(t)(-A)^{-\alpha}\| = O(1/t)$ as $t\to \infty$, then
$
\|R(\lambda,A)(-A)^{-\alpha}\|= O(|\log  \re \lambda|)
$
as $\re \lambda \to 0+$.
From this, we see that Proposition~\ref{prop:t_logt}.b) with $C \coloneqq (-A)^{-\alpha}$ does not 
give a sharp bound of $\|R(\lambda,A)(-A)^{-\alpha}\|$ in the Hilbert space setting. 
In contrast, Proposition~\ref{prop:t_logt}.a) with $C \coloneqq (-A)^{-\alpha}$,
combined with Proposition~\ref{prop:WC_to_poly_decay},
yields a sharp result on the relation between the rate of
decay of $\|T(t)(-A)^{-\alpha}\|$ as $t\to\infty$
and the rate of growth of $\|R(\lambda,A)(-A)^{-\alpha}\|$ as $\re \lambda \to 0+$.

\begin{theorem}
	\label{thm:poly_log}
	Let $(T(t))_{t \geq 0}$ be a bounded $C_0$-semigroup on a Hilbert 
	space $X$ with generator $A$ such that $0 \in  \varrho(A)$.
	The following statements are equivalent for fixed 
	$\alpha  > 0$, 
	$0 \leq \beta < 1$, and $\gamma\geq 0$:
	\begin{enumerate}[\em (\roman{enumi})]
		\item
		$\
		\|T(t)(-A)^{-\alpha}\| = O (
		t^{-\beta} (\log t)^{-\gamma}
		)
		$
		as $t \to \infty$.
		\item
		There exists $K >0$ such that
		\[
		\|R(\lambda,A)(-A)^{-\alpha}\| \leq
		\frac{K}{(\re \lambda)^{1-\beta} |\log \re \lambda|^\gamma}
		\]
		for all $\lambda \in \mathbb{C}_+$ with $\re \lambda < 1$.
	\end{enumerate}
\end{theorem}
\begin{proof}
	The implication (i) $\Rightarrow$ (ii) has already been proved in Proposition~\ref{prop:t_logt}.a)
	with $C \coloneqq (-A)^{-\alpha}$.
	To show the converse implication (ii) $\Rightarrow$ (i), we define the function $F$ in 
	Proposition~\ref{prop:WC_to_poly_decay} by
	\[
	F(\xi ) \coloneqq \frac{K}{\xi ^{1-\beta} |\log \xi |^{\gamma}} 
	\]
	for $0 < \xi  < 1$.  Then
	\[
	\frac{F(1/t)}{t} = \frac{K}{t^\beta (\log t)^{\gamma}}
	\]
	for all $t > 1$.
	Thus, the  implication (ii) $\Rightarrow$ (i)
	holds by Proposition~\ref{prop:WC_to_poly_decay} with $C \coloneqq (-A)^{-\alpha}$.
\end{proof}

We also obtain a characterization of the same class of 
non-polynomial rates of decay for
$\|T(t) (-A)^{\alpha}(I-A)^{-\alpha-\delta}\|$ with $\alpha >0$ and 
$\delta \geq 0$.
The proof is the same as that of Theorem~\ref{thm:poly_log} and hence is omitted.
\begin{theorem}
	Let $(T(t))_{t \geq 0}$ be a bounded $C_0$-semigroup on a Hilbert 
	space $X$ with generator $A$.
	The following statements are equivalent for fixed $\alpha >0$, 
	$0\leq \beta < 1$, 
	$\gamma\geq 0$, and $\delta \geq 0$:
	\begin{enumerate}[\em (\roman{enumi})]
		\item 
		$\|T(t) (-A)^{\alpha}(I-A)^{-\alpha-\delta}\| = O (
		t^{-\beta} (\log t)^{-\gamma}
		)$ as $t\to \infty$.
		\item 
		There exists $K >0$ such that
		\[
		\|R(\lambda,A)(-A)^{\alpha}(I-A)^{-\alpha-\delta}\| \leq
		\frac{K}{(\re \lambda)^{1-\beta} |\log \re \lambda|^\gamma}
		\]
		for all $\lambda \in \mathbb{C}_+$ with $\re \lambda < 1$.
	\end{enumerate}
\end{theorem}

\section{Sufficient condition for $L^2$-infinite-time admissibility}
\label{sec:suff_cond}
Let $X$ and $Y$ be Hilbert spaces.
Let $A$ be the generator of 
a bounded $C_0$-semigroup $(T(t))_{t \geq 0}$
on $X$.
In this section, we derive a sufficient condition for
$C \in \mathcal{L}(D(A,Y)$ to be $L^2$-infinite-time admissible.
Before proceeding to the details, we recall the definitions of
polynomial stability of $C_0$-semigroups and 
$L^2$-finite-time admissibility.

\begin{definition}
	{\em
		A $C_0$-semigroup $(T(t))_{t\geq 0}$
		on a Banach space $X$  with generator $A$ is {\em polynomially 
			stable} if the following two conditions are satisfied:
		\begin{enumerate}[\alph{enumi})]
			\item
			$(T(t))_{t\geq 0}$ is  bounded.
			\item
			There exists $\alpha >0$ such that $\displaystyle 
			\|T(t)(I-A)^{-1}\| = O(t^{-1/\alpha})$ as $t \to \infty$.
		\end{enumerate}
	}
\end{definition}

Let $A$ be the generator of a polynomially stable $C_0$-semigroup $(T(t))_{t\geq 0}$ on a Banach space. Then $i \mathbb{R} \cap \sigma (A) = \emptyset$ by \cite[Theorem~1.1]{Batty2008}. Therefore, when 
$\|T(t)(I-A)^{-1}\| = O(t^{-1/\alpha})$ as $t \to \infty$
for some $\alpha >0$, 
the estimate 
$\|T(t)A^{-1}\| = O(t^{-1/\alpha})$ as $t \to \infty$ is also satisfied.

\begin{definition}
	{\em
		Let $X$ and $Y$ be Banach spaces, and
		let $A$ be the generator of a $C_0$-semigroup 
		$(T(t))_{t \geq 0}$ on $X$.
		An operator $C\in \mathcal{L}(D(A),Y)$ is {\em $L^2$-finite-time admissible for $A$}
		if 
		there exist $M_1>0$ and $t_1>0$ such that 
		\[
		\int^{t_1}_0 \|CT(t)x\|^2 dt \leq M_1 \|x\|^2
		\]
		for all $x \in D(A)$.
	}
\end{definition}
From the semigroup property, it follows
that if
$C\in \mathcal{L}(D(A),Y)$ is $L^2$-finite-time admissible for $A$,
then {\em for  each} $t_2 >0$, there exists $M_2>0$ such that for all $x \in D(A)$,
\[
\int^{t_2}_0 \|CT(t)x\|^2 dt \leq M_2 \|x\|^2.
\]
By definition, every
$C \in \mathcal{L}(X,Y)$ is $L^2$-finite-time
admissible.

Let $X$ and $Y$ be Hilbert spaces, and let $\mathbb{C}_a \coloneqq \{
\lambda \in \mathbb{C}: \re \lambda > a
\}$ for $a \in \mathbb{R}$.
If $A$ is the generator of a $C_0$-semigroup on $X$
such that $T(t_0)$ is surjective for some $t_0>0$ 
and if $C \in \mathcal{L}(D(A),Y)$
satisfies
$\sup_{\lambda \in \mathbb{C}_a} \|CR(\lambda ,A)\| < \infty$ for some $a \in \mathbb{R}$,
then $C$ is $L^2$-finite-time admissible for $A$; see \cite[Theorem~4.1]{Weiss1991} and 
\cite[Theorem~2.2]{Zwart2005}.
Other sufficient conditions for $L^2$-finite-time admissibility 
in the Hilbert space setting can be found, e.g.,
in \cite[Theorem~3.3]{Zwart2005} and \cite[Lemma~2.6]{Chill2023}.

Combining $L^2$-finite-time admissibility with the
logarithmic decay of $\|T(t)(-A)^{-\alpha}\|$ and the $2$-Weiss condition
of $C(-A)^{\alpha}$, we obtain a sufficient condition for 
$L^2$-infinite-time admissibility. It
can be proved by the technique used 
in the proof of  \cite[Theorem~4.2]{Zwart2005}.
Note that 
the extension of $C(-A)^{\alpha}$ to $\mathcal{L}(D(A),Y)$ is unique
if it exists, since $D((-A)^{1+\alpha})$ is dense in $D(A)$ with respect to the graph norm
$\|\cdot\|_A$.
\begin{theorem}
	\label{thm:Weiss_cond}
	Let $X$ and $Y$ be Hilbert spaces.
	Let $(T(t))_{t\geq 0}$ be 
	a bounded $C_0$-semigroup
	on $X$ with generator $A$ such that $0 \in \varrho(A)$,
	and let $C \in \mathcal{L}(D(A),Y)$.
	Assume that the following conditions
	hold for fixed $\alpha >0$ and $\beta >1/2$:
	\begin{enumerate}[\em \alph{enumi})]
		\item $C$ is $L^2$-finite-time admissible for $A$; 
		\item
		$
		\|T(t)(-A)^{-\alpha}\| \leq 
		O((\log t)^{-\beta})
		$
		as $t \to \infty$; and
		\item 
		$C(-A)^{\alpha}$ extends to an operator in $\mathcal{L}(D(A),Y)$,
		and its extension satisfies the $2$-Weiss condition for $A$.
	\end{enumerate}
	Then
	$C$ is $L^2$-infinite-time admissible for $A$.
\end{theorem}

\begin{proof}
	Using Plancherel's theorem, we have that
	for all $\xi >0$ and $x \in X$,
	\begin{align}
	\label{eq:first_Plancherel}
	\int^{\infty}_{-\infty}
	\|R(\xi+i\eta,A)(-A)^{-\alpha} x\|^2 d\eta
	&= 2\pi 
	\int^{\infty}_0  
	\|e^{-\xi t} T(t)(-A)^{-\alpha} x\|^2 dt.
	\end{align}
	By the assumption b),
	there exist constants 
	$M_0>0$ and $t_0 >e^{2\beta}$ such that 
	\[
	\|T(t)(-A)^{-\alpha}\| \leq \frac{M_0}{(\log t)^\beta} 
	\]
	for all $t >t_0$. 
	Lemma~\ref{lem:integral_estimate} shows that
	there exists a constant $M_1 >0$
	such that
	\[
	\int^{\infty}_{t_0} \frac{e^{-\xi t}}{(\log t)^{2\beta}} dt \leq 
	\frac{M_1}{\xi |\log \xi |^{2\beta}}
	\]
	for all $0 < \xi  < 1/t_0$.
	Put $c \coloneqq \sup_{t \geq 0} \|T(t) (-A)^{-\alpha}\|$. Then
	\begin{align}
	\int^{\infty}_0  
	\|e^{-\xi  t} T(t)(-A)^{-\alpha} x\|^2 dt
	&\leq 
	c^2t_0 \|x\|^2
	+ 
	M_0^2
	\int^{\infty}_{t_0} \frac{e^{-2\xi   t}}{(\log t)^{2\beta}} dt  \|x\|^2\notag \\
	&\leq 
	\left(c^2t_0 + \frac{M_0^2 M_1}{2\xi  |\log (2\xi )|^{2\beta} }\right) \|x\|^2
	\label{eq:RAx}
	\end{align}
	for all  $0 < \xi  < 1/(2t_0)$ and $x \in X$.
	Let the extension of $C(-A)^{\alpha}$ in the assumption c) be 
	also denoted by $C(-A)^{\alpha}$.
	There exists a constant $K >0$ such that
	\begin{equation}
	\label{eq:CAR}
	\|C(-A)^{\alpha} R(\lambda,A)\|
	\leq \frac{K}{\sqrt{\re \lambda}}
	\end{equation}
	for all $\lambda \in \mathbb{C}_+$.
	From  \eqref{eq:first_Plancherel}--\eqref{eq:CAR},
	it follows that for all  $0 < \xi  < 1/(2t_0)$ and $x \in X$,
	\begin{align}
	&
	\int^{\infty}_{-\infty}
	\|CR(\xi +i\eta,A)^2x\|^2 d\eta \notag \\
	&\qquad \leq 
	\sup_{\eta \in \mathbb{R}}\|C(-A)^{\alpha} R(\xi +i\eta,A)\|^2
	\int^{\infty}_{-\infty}
	\|R(\xi +i\eta,A)(-A)^{-\alpha} x\|^2 d\eta  \notag \\
	&\qquad \leq 
	\frac{2\pi K^2}{\xi } \left(c^2t_0 + \frac{M_0^2 M_1}{2\xi |\log (2\xi )|^{2\beta} }\right)  \|x\|^2.
	\label{eq:CR2x}
	\end{align}
	Using Plancherel's theorem again, we obtain that for all $\xi  >0$ and $x \in D(A)$,
	\[
	\int^{\infty}_0 \|t e^{-\xi  t} CT(t)x\|^2 dt =
	\frac{1}{2\pi}
	\int^{\infty}_{-\infty}
	\|CR(\xi +i\eta,A)^2x\|^2 d\eta.
	\]
	This and  \eqref{eq:CR2x} show that there exists 
	a constant $M_2 >0$ such that
	\begin{equation}
	\label{eq:ateCT}
	\int^{\infty}_0 \|\xi  t e^{-\xi t} CT(t)x\|^2 dt \leq 
	\frac{M_2}{|\log(2\xi )|^{2\beta}}   \|x\|^2
	\end{equation}
	for all $0 < \xi  < 1/(2t_0)$ and $x \in D(A)$.
	
	Take $\tau_1 < 1 < \tau_2$ such that $\tau_1 e^{-\tau_1} = 1/(2e) = 
	\tau_2 e^{-\tau_2}$. Set 
	\[
	\mu_n \coloneqq
	\left(
	\frac{\tau_1}{\tau_2}
	\right)^{n-1},\quad \tau_n \coloneqq \frac{\tau_1}{\mu_n}.
	\]
	Then $\mu_n \to 0$ and  $\tau_n \to \infty$ as $n \to \infty$.
	Moreover, for all $n \in \mathbb{N}$,
	\begin{equation}
	\label{eq:mue_lower_bound}
	\mu_n t e^{-\mu_n t} > \frac{1}{2e}
	\end{equation}
	whenever $t \in (\tau_n,\tau_{n+1})$.
	Let $m \in \mathbb{N}$ satisfy $\mu_m < 1/(2t_0)$.
	Combining the estimates \eqref{eq:ateCT} and \eqref{eq:mue_lower_bound},
	we have that 
	for all $x \in D(A)$,
	\begin{align*}
	\int^{\infty}_{\tau_m} \|CT(t)x\|^2 dt &=
	\sum_{n=m}^{\infty} \int^{\tau_{n+1}}_{\tau_n} \|CT(t)x\|^2 dt \\
	&\leq 
	(2e)^2\sum_{n=m}^{\infty} \int^{\tau_{n+1}}_{\tau_n} \|\mu_n t e^{-\mu_n t} CT(t)x\|^2 dt  \\
	&\leq 
	(2e)^2 M_2\|x\|^2 \sum_{n=m}^{\infty} \frac{1}{|\log (2\mu_n)|^{2\beta}}.
	\end{align*}
	Since 
	\[
	\frac{1}{|\log (2\mu_n)|}=
	\frac{1}{(\log(\tau_2) - \log(\tau_1) )(n-1) - \log 2 }
	\]
	for all $n \geq m$,
	we have from $\beta > 1/2$ that
	\[
	M_3 \coloneqq 
	(2e)^2 M_2  \sum_{n=m}^{\infty} \frac{1}{|\log (2\mu_n)|^{2\beta}}
	\in (0,\infty),
	\]
	and then
	\begin{equation}
	\label{eq:CT_tau1_bound}
	\int^{\infty}_{\tau_m} \|CT(t)x\|^2 dt \leq M_3 \|x\|^2
	\end{equation}
	for all $x \in D(A)$.
	Combining the estimate \eqref{eq:CT_tau1_bound} 
	with the finite-time admissibility of $C$ in the assumption a), we conclude that 
	$C$ is infinite-time admissible for $A$.
\end{proof}

The assumption c) of Theorem~\ref{thm:Weiss_cond}
leads us to the following strong version of the Weiss condition.
\begin{definition}
	{\em
		Let $X$ and $Y$ be Banach spaces, and
		let $A$ be the generator of a bounded $C_0$-semigroup on $X$.
		An operator $C\in \mathcal{L}(D(A),Y)$ satisfies the 
		{\em strong $2$-Weiss condition
			for $A$} 
		if there exists $\alpha >0$ such that 
		$C(-A)^{\alpha}$ extends to an operator in $\mathcal{L}(D(A),Y)$ 
		and its extension satisfies the $2$-Weiss condition for $A$.
	}
\end{definition}

By Proposition~\ref{prop:interpolation}, if $A$ is the generator of a
polynomially stable $C_0$-semigroup $(T(t))_{\geq 0}$ on a Banach space, 
then for all $\alpha >0$, 
there exists $\beta >0$ such that 
$\|T(t)(-A)^{-\alpha}\| = O(t^{-\beta})$ as $t \to \infty$,
which is a faster decay rate than the one of the assumption b) in Theorem~\ref{thm:Weiss_cond}.
Therefore, we obtain the following result as a direct consequence of
Theorem~\ref{thm:Weiss_cond}.
\begin{corollary}
	\label{coro:Weiss_cond}
	Let $X$ and $Y$ be Hilbert spaces, and
	let $A$ be the generator of 
	a polynomially stable $C_0$-semigroup on $X$. 
	If $C \in \mathcal{L}(D(A),Y)$
	is $L^2$-finite-time admissible for $A$ and satisfies the 
	strong $2$-Weiss condition for $A$,
	then
	$C$ is $L^2$-infinite-time admissible for $A$.
\end{corollary}

The case $X=Y= \ell^2(\mathbb{N})$ is 
simple but practically important as explained in \cite[Remark~2.7]{Hansen1997}.
Consider a diagonal operator $A$ on $\ell^2(\mathbb{N})$.
The resolvent condition
$\|C(-A)^{\alpha}R(\lambda, A)\| \leq K/\sqrt{\re \lambda}$ in
the strong $2$-Weiss condition is transformed into the
operator Carleson measure criterion introduced in 
\cite[Definition~1.1]{Hansen1991}.
To describe the equivalence more precisely,
let $(\lambda_n)_{n\in \mathbb{N}}$  be the eigenvalues
of the 
diagonal operator $A$ on $\ell^2(\mathbb{N})$. 
We assume that $\re \lambda_n < 0$ for all $n \in \mathbb{N}$.
Let $C \in \mathcal{L}(D(A), \ell^2(\mathbb{N}) )$ and put
$c_n \coloneqq Ce_n \in \ell^2(\mathbb{N})
$ 
for $n \in \mathbb{N}$, where $e_n$ 
is the $n$th unit vector in 
$\ell^2(\mathbb{N})$. 
Define $c_n^*\colon
\ell^2(\mathbb{N}) \to \mathbb{C}$ by
$c_n^*x = \langle x,c_n \rangle$ for $x \in \ell^2(\mathbb{N})$.
When the operator $C$ is represented by an infinite matrix,
$c_n$ is the $n$th column of $C$, and $c_nc_n^*$ is
an infinite matrix of rank one.
For $h>0$ and $\omega \in \mathbb{R}$,
the rectangle $Q(h,\omega )$ is given by
\[
Q(h,\omega ) \coloneqq \{
\lambda \in \mathbb{C}:
0\leq \re \lambda \leq h,\,
|\im \lambda - \omega | \leq h
\}.
\]
For a fixed $\alpha > 0$,
\cite[Proposition~5.3]{Hansen1997} shows that 
$C(-A)^{\alpha}$ extends to an operator in $\mathcal{L}(D(A),\ell^2(\mathbb{N}))$
and
there exists
$K>0$ such that
$\|C(-A)^{\alpha}R(\lambda, A)\| \leq K/\sqrt{\re \lambda}$ for 
all $\lambda \in \mathbb{C}_+$ if and only if
there exists $M>0$ such that for  all $h>0$ and $\omega  \in \mathbb{R}$, 
\[
\left\|
\sum_{-\lambda_n \in Q(h,\omega )} |\lambda_n|^{2\alpha} c_nc_n^*
\right\| \leq Mh,
\]
where the 
norm on the left-hand side is the operator norm on $\ell^2(\mathbb{N})$.

We have seen in 
Lemma~\ref{lem:poly_decay_to_WC} 
and Proposition~\ref{prop:WC_to_poly_decay} that 
when $T(t)$ commutes with $C \in \mathcal{L}(D(A),X)$ for all $t \geq 0$,
the estimate
$\|CT(t)\| = O(1/\sqrt{t})$ as $t \to \infty$ is equivalent to 
the $2$-Weiss condition on $C$ in the Hilbert space setting.
The next proposition shows that 
the strong $2$-Weiss condition implies a slightly better decay rate
$\|CT(t)\| = O(1/\sqrt{t^{1+\beta}})$ for some $\beta >0$ if
$(T(t))_{t \geq 0}$ is polynomially stable.

\begin{proposition}
	\label{prop:qWeiss_decay}
	Let $A$ be the generator of
	a polynomially stable $C_0$-semigroup $(T(t))_{t \geq 0}$
	on a Hilbert space $X$.
	Let  $C \in \mathcal{L}(D(A),X)$ be such that  $T(t)$ commutes with $C$ 
	for all $t \geq 0$. If $C$ 
	satisfies the strong $2$-Weiss condition for $A$, then
	the operator $CT(t)$ extends to
	a bounded linear operator (also denoted by $CT(t)$) on $X$ for all $t >0$, 
	and
	there exist $M>0$ and $\beta > 0$ such that 
	\begin{equation}
	\label{eq:CT_decay}
	\|CT(t)\| \leq
	\frac{M}{\sqrt{t^{1+\beta}}}
	\end{equation}
	for all $t >0$.
\end{proposition}
\begin{proof}
	Since $C$ satisfies the strong $2$-Weiss condition for $A$, 
	there exists $\alpha >0$ such that
	$C(-A)^{\alpha }$ 
	extends to an operator in $\mathcal{L}(D(A),X)$
	and its extension satisfies the $2$-Weiss condition for $A$.
	By assumption, $T(t)$ commutes $C(-A)^{\alpha }$ and hence its extension 
	for all $t \geq 0$. Then it follows from
	Proposition~\ref{prop:WC_to_poly_decay}  that there is $M_1 >0$
	such that
	\begin{equation}
	\label{eq:CAT_bound}
	\|C(-A)^{\alpha } T(t)x \| \leq \frac{M_1}{\sqrt{t}}\|x\|
	\end{equation}
	for all $t >0$ and 
	$x \in D((-A)^{1+\alpha })$.
	Since $(T(t))_{t \geq 0}$ is polynomially stable,
	Proposition~\ref{prop:interpolation} shows
	that there exist $M_2 >0$ and $\beta >0$ such that
	\begin{equation}
	\label{eq:TA_bound}
	\| T(t) (-A)^{-\alpha }\| \leq \frac{M_2}{\sqrt{t^\beta}}
	\end{equation}
	for all $t >0$.
	Combining the inequalities \eqref{eq:CAT_bound} and
	\eqref{eq:TA_bound}, we obtain
	\begin{align*}
	\|CT(t)x\| &=
	\|C(-A)^{\alpha } T(t/2) T(t/2) (-A)^{-\alpha }x\| \\
	&\leq 
	\frac{M_1}{\sqrt{t/2}} \,
	\| T(t/2) (-A)^{-\alpha }x\| \\
	&\leq \frac{\sqrt{2^{1+\beta}}M_1M_2}{\sqrt{t^{1+\beta}}}
	\|x\|
	\end{align*}
	for all $t >0$ and 
	$x \in D(A)$.
	The result follows by the density of $D(A)$ in $X$.
\end{proof}

We also show that 
the decay property $\|CT(t)\| = O(1/\sqrt{t^{1+\beta}})$ for some $\beta >0$ 
leads to the strong $2$-Weiss condition
under some additional assumption on boundedness.

\begin{proposition}
	\label{prop:decay_qWeiss}
	Let $A$ be the generator of
	a bounded $C_0$-semigroup $(T(t))_{t \geq 0}$
	on a Banach space $X$.
	Let  $C \in \mathcal{L}(D(A),X)$ be such that  $T(t)$ commutes with $C$ 
	for all $t \geq 0$.
	If there exist constants $M_1,M_2, \alpha,\beta >0$ such that for all $t >0$,
	\begin{equation}
	\label{eq:st_Weiss_cond_bound1}
	\|C(-A)^\alpha T(t)x\| \leq M_1 \|x\|,\quad x \in D((-A)^{1+\alpha})
	\end{equation}
	and
	\begin{equation}
	\label{eq:st_Weiss_cond_bound2}
	\|CT(t)x\| \leq \frac{M_2}{\sqrt{t^{1+\beta}}} \|x\|,\quad x \in D(A),
	\end{equation}	
	then
	$C$ satisfies the strong $2$-Weiss condition for $A$.
\end{proposition}
\begin{proof}
	By assumption, $\widetilde C \coloneqq C(I-A)^{-1} \in \mathcal{L}(X)$ commutes with
	$A$. Let $\delta >0$. It follows from
	\cite[Proposition~3.1.1.f)]{Haase2006}
	that $\widetilde C$ also commutes with $(-A)^{\delta }$.
	For all $x \in D((-A)^{1+\delta }) = D((I-A)^{1+\delta }) $,
	there exists $y \in D((-A)^{\delta }) = D((I-A)^{\delta})$ such that $x = (I-A)^{-1} y$. 
	Then
	\begin{align}
	\label{eq:CA_frac_comm}
	C (-A)^{\delta }x = 
	\widetilde C (-A)^{\delta } y =  (-A)^{\delta }\widetilde C y =
	(-A)^{\delta }C x.
	\end{align}

	Define $\gamma \coloneqq \alpha \beta/ (1+\beta) \in (0,\alpha)$.
	The moment inequality (see, e.g., \cite[Proposition~6.6.4]{Haase2006}) shows that
	there exists a constant $c>0$ such that 
	for all $t \geq 0$ and $x \in D((-A)^{1+\alpha})$.
	\begin{equation}
	\label{eq:st_Weiss_cond_bound3}
	\|(-A)^\gamma CT(t)x \| \leq 
	c \|(-A)^\alpha CT(t)x\|^{\gamma/\alpha} \,\|C T(t)x\|^{1 - \gamma/\alpha}.
	\end{equation}
	Since 
	\[
	\frac{1+\beta}{2}\left(
	1- \frac{\gamma}{\alpha}
	\right) = \frac{1}{2},
	\]
	we have from 
	\eqref{eq:st_Weiss_cond_bound1}--\eqref{eq:st_Weiss_cond_bound3}
	that
	\begin{align}
	\label{eq:A_gamma_estimate}
	\|C(-A)^\gamma T(t)x \| \leq
	\frac{M_3}{
		\sqrt{t}
	} \|x\|
	\end{align}
	for all $t >0$ and $x \in D((-A)^{1+\alpha})$,
	where $M_3 \coloneqq c M_1^{\gamma/\alpha}M_2^{1-\gamma/\alpha}$.

	The estimate \eqref{eq:A_gamma_estimate} yields that
	for all $\lambda \in \mathbb{C}_+$ and $x \in D((-A)^{1+\alpha})$,
	\begin{align}
	\|
	C(-A)^\gamma R(\lambda,A)x
	\| &\leq 
	\int^\infty_0
	e^{- t \re \lambda}
	\|C (-A)^\gamma  T(t) x\| dt \notag \\
	&\leq 
	M_3 \int^{\infty}_0 
	\frac{e^{-t \re \lambda }}{\sqrt{t}} dt
	\|x\| \notag\\
	&=
	\frac{M_3\Gamma(1/2) }{\sqrt{\re \lambda}} \|x\|,
	\label{eq:CAR_bound}
	\end{align}
	where $\Gamma$ is the gamma function.
	By 
	\cite[Proposition~3.1.1.h)]{Haase2006}, 
	$D((-A)^{1+\alpha})$ is a core for $(-A)^\gamma$, and hence
	the estimate \eqref{eq:CAR_bound} holds for all $x \in D((-A)^\gamma)$.
	Since $D((-A)^{\gamma})$ is dense in $X$, it follows that 
	$C(-A)^\gamma R(\lambda,A)$ extends to a bounded linear
	operator on $X$ for a fixed $\lambda \in \mathbb{C}_+$. Therefore, $C(-A)^\gamma$ extends to an
	operator in $\mathcal{L}(D(A),X)$.
	We also see from the estimate \eqref{eq:CAR_bound} that 
	the extension of
	$C(-A)^\gamma$ satisfies the $2$-Weiss condition for $A$.
	Thus, $C$ satisfies the strong $2$-Weiss condition for $A$.
\end{proof}

\end{document}